\documentclass[11pt, leqno]{amsart}

\usepackage[utf8]{inputenc} 

\usepackage[colorlinks=true, linkcolor=blue, citecolor=blue]{hyperref}

\usepackage{amssymb}
\usepackage{amsmath, graphicx, rotating}
\usepackage{color}
\usepackage{soul}
\usepackage[dvipsnames]{xcolor}

\usepackage[T1]{fontenc}
\usepackage{lmodern}
\usepackage[english]{babel}

\usepackage{ upgreek }
\usepackage{stmaryrd}
\SetSymbolFont{stmry}{bold}{U}{stmry}{m}{n}
\usepackage{amsthm}
\usepackage{float}

\usepackage{ bbm }
\usepackage{ stmaryrd } 
\usepackage{ mathrsfs }
\usepackage{ frcursive }
\usepackage{ comment }

\usepackage{pgf, tikz}
\usetikzlibrary{shapes}
\usepackage{varioref}
\usepackage{enumitem}

\usepackage[textheight=7.61in,textwidth=4.98in]{geometry} 
\newtheorem{theorem}{Theorem}[section]
\newtheorem*{theorem*}{Theorem}
\newtheorem{lemma}[theorem]{Lemma}

\newtheorem{proposition}[theorem]{Proposition}

\labelformat{hypothesis}{\textbf{M\kern-0.1mm#1}}

\newtheorem{condition}{Condition}

\newcommand{\mycomment}[1]{}

\theoremstyle{definition}

\numberwithin{equation}{section}

\newcommand*{\abs}[1]{\left\lvert#1\right\rvert}
\newcommand*{\norm}[1]{\left\lVert#1\right\rVert}

\newcommand*{\sachant}[2]{\left.#1 \,\middle|\,#2\right.}

\def\bb#1{\mathbb{#1}}
\def\bs#1{\boldsymbol{#1}}
\def\bf#1{\mathbf{#1}}
\def\scr#1{\mathscr{#1}}

\def\geq{\geqslant}
\def\leq{\leqslant}
\def\phi{\varphi}

\newcommand\ee{\varepsilon}
\renewcommand\ll{\lambda}

\DeclareMathOperator{\dd}{d\!}
\DeclareMathOperator{\e}{e}

\DeclareMathOperator{\supp}{supp}

\begin{document}

\title[Branching processes in Markovian environment]
{Limit theorems  for critical branching processes in a finite state space Markovian environment}

\author{Ion Grama}
\curraddr[Grama, I.]{ Universit\'{e} de Bretagne-Sud, LMBA UMR CNRS 6205, Vannes, France}
\email{ion.grama@univ-ubs.fr}

\author{Ronan Lauvergnat}
\curraddr[Lauvergnat, R.]{Universit\'{e} de Bretagne-Sud, LMBA UMR CNRS 6205,
Vannes, France}
\email{ronan.lauvergnat@univ-ubs.fr}

\author{\'Emile Le Page}
\curraddr[Le Page, \'E.]{Universit\'{e} de Bretagne-Sud, LMBA UMR CNRS 6205,
Vannes, France}
\email{emile.le-page@univ-ubs.fr}

\begin{abstract}
Let $(Z_n)_{n\geq 0}$ be a critical branching process 
in a random environment 
defined by a Markov chain $(X_n)_{n\geq 0}$ 
with values in a finite state space $\bb X$.
Let $ S_n = \sum_{k=1}^n \ln  f_{X_k}'(1)$ 
be the Markov walk associated to $(X_n)_{n\geq 0}$, where 
$f_i$ is the offspring generating function when the environment is $i \in \bb X$.
Conditioned on the event $\{ Z_n>0\}$, we show 
the non degeneracy of limit law of the normalized number of particles ${Z_n}/{e^{S_n}}$ 
and determine the limit of the law of $\frac{S_n}{\sqrt{n}} $ jointly with $X_n$.    
Based on these results we establish a Yaglom-type theorem 
which specifies the limit of the joint law of $ \log Z_n$ and  $X_n$ given $Z_n>0$.
\end{abstract}

\date{\today}

\maketitle

\section{Introduction and main results} \label{sec not res}
One of the most used models in the dynamic of populations is 
the Galton-Watson branching process which
has numerous applications in different areas such as physics, biology, medicine, economics etc. 
We refer the reader to  the books of
Harris \cite{harris2002theory} and Athreya and Ney \cite{athreya_branching_1972} for an introduction.
Branching processes in random environment have been first  considered by 
Smith and Wilkinson \cite{smith_branching_1969}, 
and Athreya and Karlin \cite{athreya1971branching1, athreya1971branching2}.
This subject has been further studied by 
Kozlov \cite{kozlov_asymptotic_1977, kozlov_1995},  
Dekking \cite{dekking_survival_1987},
Liu \cite{liu1996survival},
D'Souza and Hambly \cite{dsouza_survival_1997}, 
Geiger and Kersting \cite{geiger_survival_2001}, 
Guivarc'h and Liu \cite{guivarch_proprietes_2001}, 
Geiger, Kersting and Vatutin \cite{geiger_limit_2003},
Afanasyev \cite{afanasyev_limit_2009}, 
Kersting and Vatutin \cite{kersting_vatutin_2017},
 to name only a few.
Recently, based on new conditioned limit theorems for sums of functions defined on Markov chains in
\cite{GLLP_affine_2016,grama_limit_2016-1,
GLLP_CLLT_2017,grama_conditioned_2016},
the  exact asymptotic results for the survival probability when the environment is a Markov chain
have been obtained 
for branching processes in Markovian environment (BPME) in 
\cite{GrLauvLePage_2018-SPA}.
In this paper we shall  complement them by new results, such as a limit theorem for the normalized number of particles and an Yaglom-type theorem for BPME. 
 
We start by introducing the Markovian environment which is given
on the probability space $\left( \Omega, \scr F, \bb P \right)$  
by a homogeneous Markov chain $\left( X_n \right)_{n\geq 0}$ 
with values in the finite state space $\bb X$ 
and with the matrix of transition probabilities $\bf P = (\bf P (i,j))_{i,j\in \mathbb X}$. 
We suppose the following:
\begin{condition}
\label{primitif}
The Markov chain $\left( X_n \right)_{n\geq 0}$ is irreducible and aperiodic.
\end{condition}
Condition \ref{primitif} implies a spectral gap property for the transition operator 
$\bf P$ of $\left( X_n \right)_{n\geq 0}$, 
 defined by the relation $\bf P g (i) = \sum_{j\in \bb X} g(j) \bf P(i,j)$ for any $g$ 
in the space $\scr C (\mathbb X)$ of complex functions $g$ on $\bb X$ endowed with the norm
$\norm{g}_{\infty} = \sup_{x\in \bb X} \abs{g(x)}$. 
Indeed, Condition \ref{primitif} is necessary and sufficient  for 
the matrix $(\bf P (i,j))_{i,j\in \mathbb X}$ to be primitive (all entries of $\bf P^{k_0}$ are positive for some $k_0 \geq 1$).
By Perron-Frobenius theorem,
there exist positive constants $c_1$, $c_2$, a unique positive $\bf P$-invariant probability 
$\bs \nu$ on $\bb X$ ($\bs \nu(\mathbf P)=\bs\nu$) and an operator $Q$ on $\scr C(\mathbb X)$ such that, for any 
$g \in \scr C(\mathbb X)$ and $n \geq 1$, $i \in \bb X$,
\begin{align} \label{labSpectrGap}
\bf Pg(i) = \bs \nu(g) + Q(g)(i) \quad \text{and} \quad \norm{Q^n(g)}_{\infty} \leq c_1\e^{-c_2n} \norm{g}_{\infty},
\end{align}
where $Q \left(1 \right) = 0$ and $\bs \nu \left(Q(g) \right) = 0$ with 
$\bs \nu(g) := \sum_{i \in \bb X} g(i) \bs \nu(i)$. 
In particular, from \eqref{labSpectrGap}, it follows that, for any $(i,j) \in \bb X^2$,
\begin{equation}	\label{EQ:exprate001}
\abs{\bf P^n(i,j) - \bs \nu(j)} \leq c_1\e^{-c_2 n}.
\end{equation}
Set $\mathbb N := \left\{0,1,2,\ldots\right\}$. 
For any $i\in \mathbb X$, let $\bb P_i$ be the probability law on $ \mathbb X^{\mathbb N} $ 
and $\bb E_i$ the associated expectation
generated by the finite dimensional distributions 
of the Markov chain $\left( X_n \right)_{n\geq 0}$ starting at $X_0 = i$.
Note that  
$\bf P^ng(i) = \bb E_i \left( g(X_n) \right)$, 
for any $g \in \scr C(\mathbb X)$, $i \in \bb X$ and $n\geq 1.$

 Assume that on the same probability space 
$\left( \Omega, \scr F, \bb P \right)$,
for any $i \in \bb X$, 
we are given 
a random variable $\xi_i$ with  the probability generating function 
\begin{equation}
	\label{jazz}
	f_i(s) := \bb E \left( s^{\xi_i} \right), \quad s \in [0,1].
\end{equation}
Consider a collection of independent and identically distributed 
random variables $( \xi_i^{n,j} )_{j,n \geq 1}$ having the same law  
as the generic variable $\xi_i$.
The variable $\xi_{i}^{n,j}$ represents the number of children generated by the parent $j\in \{1, 2, \dots\}$  
at time $n$ when the environment  is $i$. 
Throughout the paper, the sequences $( \xi_i^{n,j} )_{j,n \geq 1}$, $i\in \bb X$ 
and the Markov chain $\left( X_n \right)_{n\geq 0}$ are supposed to be independent.

Denote by $\mathbb E$ the expectation associated to $\mathbb P.$ 
We assume that the variables $\xi_i$ have a positive means and finite second moments.
\begin{condition}
\label{Cond moments}
For any $i \in \bb X$, the random variable $\xi_i$ satisfies the following:
$\bb E \left( \xi_i \right)>0$ and $\bb E ( \xi_i^2 ) < +\infty.$
\end{condition}
From Condition \ref{Cond moments} it follows that $0< f_i'(1) < +\infty$ and $f_i''(1) < +\infty.$ 

We are now prepared to introduce the branching process $\left( Z_n \right)_{n\geq 0}$
in the Markovian environment $\left( X_n \right)_{n\geq 0}$.
The initial population size is $Z_0 = z \in \mathbb N$.  
For $n\geq 1$, we let $Z_{n-1}$ be the population size at time $n-1$ and
assume that at time $n$
the parent $j\in \{1, \dots Z_{n-1} \}$ 
generates $\xi_{X_n}^{n,j}$ children.
Then the population size at time $n$ is given by
\begin{equation*}
\label{roseau}
Z_{n} = \sum_{j=1}^{Z_{n-1}} \xi_{X_{n}}^{n,j}, 
\end{equation*}
 where  the empty sum is equal to $0$.
In particular, when $Z_0=0$, it follows that $Z_n=0$ for any $n\geq 1$. 
 We note that for any $n\geq 1$ the variables $\xi_{i}^{n,j}$, $j\geq 1$,  $i\in \mathbb X$ 
are independent 
of $Z_{0},\ldots,Z_{n-1}$. 

Introduce the function $\rho: \mathbb X \mapsto \mathbb R$ satisfying 
\begin{equation*}
	\rho(i) = \ln  f_i'(1) , \quad i \in \bb X.
\end{equation*}
Along with  $\left( Z_n \right)_{n\geq 0}$ consider the Markov walk $\left( S_n \right)_{n\geq 0}$ 
such that $S_0 = 0$ and, for $n\geq 1$, 
\begin{equation}\label{petale}
S_n = \ln \left( f_{X_1}'(1) \cdots f_{X_n}'(1) \right)
= \sum_{k=1}^n \rho\left( X_k \right).
\end{equation} 

The couple $\left( X_n, Z_n\right)_{n\geq 0}$ 
is a Markov chain with 
values in $\mathbb X \times \mathbb N $, whose transition operator $\widetilde{\mathbf P} $ is defined 
 by the following relation:
for any $i\in \mathbb X$, $z\in \mathbb N,$ $s\in [0,1]$ 
and $h: \mathbb X  \mapsto  \mathbb R$ bounded measurable,
\begin{align} \label{transprob001} 
\widetilde{\mathbf  P} (h_s)(i,z) = \sum_{j \in \mathbb X} \mathbf P(i,j) h(j) [f_j(s)]^z, 
\end{align}
where $h_s(i,z)= h(i) s^z$.
Let $\bb P_{i,z}$ be the probability law on $ (\mathbb X \times \mathbb N )^{\mathbb N}$
 and $\bb E_{i,z}$ the associated expectation 
generated by the finite dimensional distributions 
of the Markov chain $\left( X_n, Z_n \right)_{n\geq 0}$ starting at $X_0 = i$ and $Z_0=z$.
By straightforward  calculations, for any $i\in \mathbb X$, $z\in \mathbb N$, 
\begin{align} \label{wwzz}
\mathbb E_{i,z} (Z_n) = z \mathbb E_{i} (e^{S_n}).
\end{align}

The following non-lattice condition 
is used indirectly in the proofs of the present paper; 
it is needed to ensure that the local limit theorem for the Markov walk \eqref{petale} holds true. 
\begin{condition}
\label{Cond non-latice}
For any $\theta, a \in \bb R$, there exist $m\geq 0$ and a path $x_0, \dots, x_m$ in $\bb X$ such that
$\bf P(x_0,x_1) \cdots \bf P(x_{m-1},x_m) \bf P(x_m,x_0) > 0$
and
$$
\rho(x_0) + \cdots + \rho(x_m) - (m+1)\theta \notin a\bb Z.
$$
\end{condition}

 Condition 3 is an extension of the corresponding non-lattice condition for 
independent and identically distributed random variables $X_0, X_1, \ldots $, which can be stated as follows: 
there exists $m\geq 0$ such that  $X_0 + \cdots + X_m$ 
does not takes values in the lattice $(m+1)\theta + a\bb Z$ with some positive probability, 
whatever $\theta, a \in \bb R$. 
Usually, the latter is formulated in an equivalent way with $m=0$.
For Markov chains, Condition \ref{Cond non-latice} is equivalent to the condition that the 
Fourier transform operator 
\begin{equation}
	\label{Fourier-transfoper}
	\bf P_{i t}g(i) := \bf P\left( \e^{i t  \rho} g \right)(i) = \bb E_{i} \left( \e^{it S_1} g(X_1) \right), 
	\quad g \in \scr C(\mathbb X), i \in \bb X,
\end{equation}
has a spectral radius strictly less than $1$ for $t\not= 0$, see Lemma 4.1 of \cite{GLLP_CLLT_2017}. 
Non-latticity for Markov chains with not necessarily finite state spaces is considered, for instance, in Shurenkov \cite{shur_1984} 
and Alsmeyer \cite{alsmeyer1994}.  

For the following facts and definitions we refer to \cite{GrLauvLePage_2018-SPA}. 
Under Condition \ref{primitif}, 
from the spectral gap property of the operator $\mathbf P$  
it follows that,  
for any $\ll \in \bb R$ and any $i \in \bb X$, 
the limit
\begin{align*}
k(\ll) := \lim_{n\to +\infty} \bb E_{i}^{1/n} \left( \e^{\ll S_n} \right)
\end{align*}
exists and does not depend on the initial state of the Markov chain $X_0=i$.
Moreover, the number $k(\lambda)$ is the spectral radius of  the transfer operator $\bf P_{\ll}$: 
\begin{equation}
	\label{transfoper}
	\bf P_{\ll}g(i) := \bf P\left( \e^{\ll \rho} g \right)(i) = \bb E_{i} \left( \e^{\ll S_1} g(X_1) \right), 
	\quad g \in \scr C(\mathbb X), i \in \bb X.
\end{equation}
In particular, under Conditions \ref{primitif} and \ref{Cond non-latice},  
$k(\lambda)$ is a simple eigenvalue of the operator $\mathbf P_{\ll}$
and there is no other eigenvalue of modulus $k(\lambda)$. 
In addition, the function $k(\ll)$ is analytic on $\mathbb R.$ 

The branching process in Markovian environment is said to be \textit{subcritical} if $k'(0)<0$, \textit{critical} if $k'(0)=0$ and \textit{supercritical} if $k'(0)>0$.
The following identity, has been established in \cite{GrLauvLePage_2018-SPA}: 
\begin{equation}
\label{classifiid}
k'(0) = \bs \nu(\rho) = \bb E_{\bs \nu} \left( \rho(X_1) \right) = \bb E_{\bs \nu} \left(  \ln f_{X_1}'(1) \right)
=\phi' (0),
\end{equation}
where $\mathbb E_{\bs \nu} $ is the expectation generated by the finite dimensional distributions 
of the Markov chain $\left( X_n \right)_{n\geq 0}$ in the stationary regime and 
$\phi
(\ll)=\bb E_{\bs \nu} 
( \exp\{\ll\ln  f_{X_1}'(1)\} )$, $\ll \in \bb R.$
Relation \eqref{classifiid}  proves that the classification made in the case of branching processes
with Markovian environment 
and that for independent and identically distributed 
 environment are coherent:   
when the random variables $\left( X_n \right)_{n\geq 1}$ 
are i.i.d.\ with common law $\bs \nu$, from \eqref{classifiid} it follows that the two classifications coincide.

In the present paper we will focus on the critical case: $k'(0)=0$.
Our first result establishes the exact asymptotic of the survival probability of $Z_n$ jointly with the event $\{X_n=j\}$ when the branching process starts with $z$ particles.
\begin{theorem}
\label{Tsurvival001} 
Assume Conditions \ref{primitif}-\ref{Cond non-latice} and
$k'(0)  = 0.$
Then, there exists a positive function $u(i,z): \bb X \times \bb N \mapsto \mathbb R_{+}^{*}$ such that for any $(i,j) \in \bb X^2$ 
and $z \in \bb N$,  $z\not=0,$
\[
\bb P_{i,z} \left( Z_n > 0 \,,\, X_n = j \right) \underset{n \to +\infty}{\sim} \frac {u(i,z) \bs \nu (j) } {\sqrt{n}}.
\]
\end{theorem}
An explicit formula for $u(i,z)$ is given in Proposition \ref{prop-converg y 001}. 
In the case $z=1$, Theorem \ref{Tsurvival001}  has been proved in \cite[Theorem 1.1]{GrLauvLePage_2018-SPA}.
The proof for the case $z>1$, which is not a direct consequence of the case $z=1$, will be given in Proposition \ref{prop-converg y 001}. 

We shall complement the previous statement by studying the asymptotic behavior of $Z_n$ given $Z_n>0$
under the following condition:
\begin{condition} \label{cond Kersting2017}
The random variables $\xi_i$, $i\in \mathbb X$ satisfy:
\begin{align*}
\inf_{ i\in\mathbb X} \mathbb P (\xi_i \geq 2) >0.  
\end{align*}
\end{condition}

Condition \ref{cond Kersting2017} is quite natural — 
it tells that each parent can generate more than 1 child with positive probability.  
In the present paper is used to prove the non-degeneracy of the limit
of the martingale $(\frac{Z_{n}}{e^{S_{n}}})_{n\geq0}$ 
in key Lemma \ref{lemma bound e power}.  

The next result concerns the non degeneracy of the limit law of the properly normalized number of particles $Z_n$ at time $n$
jointly with the event $\{X_n=j\}$.
\begin{theorem}
\label{T-Loi-limite001}
Assume Conditions \ref{primitif}-\ref{cond Kersting2017} and
$k'(0)  = 0.$
Then, 
for any $i \in \bb X$, 
$z \in \bb N$,  $z\not=0,$ 
there exists a probability measure $\mu_{i,z}$ on $\bb R_{+}$ such that,
for any continuity point $t\geq0$ of the distribution function $\mu_{i,z} ([0,\cdot])$
and $j \in \bb X$, 
it holds that
$$
\lim_{n\to\infty} \sqrt{n} \bb P_{i,z} \left(  \frac{Z_n}{e^{S_n}}  \leq t, X_n = j, Z_n>0 \right) = \mu_{i,z} ([0,t]) \bs \nu (j) u(i,z)
$$
and
$$
\lim_{n\to\infty} \bb P_{i,z} \left(  \frac{Z_n}{e^{S_n}}  \leq t, X_n = j \big| Z_n>0 \right) = \mu_{i,z} ([0,t]) \bs \nu (j).
$$
Moreover, it holds that $\mu_{i,z} (\{0\})=0$.
\end{theorem}


From \cite[Lemma 10.3]{GLLP_CLLT_2017} it follows that, under Conditions \ref{primitif} and \ref{Cond non-latice},
the quantity 
\begin{equation}
\label{comete}
\sigma^2 := \bs \nu \left( \rho^2 \right) -  \bs \nu \left( \rho \right)^2 + 2 \sum_{n=1}^{+\infty} \left[ \bs \nu \left( \rho \bf P^n \rho \right) -  \bs \nu \left( \rho \right)^2 \right]
\end{equation}
is finite and positive, i.e.\ $0< \sigma <\infty$. 
Let  
 $$\Phi^{+} (t) =  (1 - e^{-\frac{t^2}{2}}) \mathbbm 1(t\geq 0),\ \ t\in \mathbb R,$$
be the Rayleigh distribution function. 
The following assertion gives the asymptotic behavior of the normalized Markov walk $S_n$
jointly with $X_n$ provided $Z_n>0$.
\begin{theorem}
\label{T-Loi-limite003}
Assume Conditions \ref{primitif}-\ref{cond Kersting2017} and $k'(0)  = 0.$
Then, 
for any $i,j \in \bb X$, 
$z \in \bb N$,  $z\not=0$
and $t \in \mathbb R$ ,
$$
\lim_{n\to\infty} \sqrt{n} \bb P_{i,z} 
\left(  \frac{ S_n}{\sigma\sqrt{n}} \leq t, X_n = j,  Z_n>0 \right) 
= \Phi^{+} (t) \bs \nu (j) u(i,z)
$$
and
$$
\lim_{n\to\infty}  \bb P_{i,z} 
\left(  \frac{ S_n}{\sigma\sqrt{n}} \leq t, X_n = j \big|  Z_n>0 \right) 
= \Phi^{+} (t) \bs \nu (j).
$$
\end{theorem}

The following assertion is the Yaglom-type limit theorem for $\log Z_n$
jointly with $X_n$.
\begin{theorem}
\label{T-Yaglom}
Assume Conditions \ref{primitif}-\ref{cond Kersting2017} and
$k'(0)  = 0.$
Then, 
for any $i \in \bb X$, 
$z \in \bb N$,  $z\not=0,$ 
$j \in \bb X$
and $t \geq 0$ ,
$$
 \lim_{n\to\infty} \sqrt{n} \bb P_{i,z} 
\left(  \frac{\log Z_n}{\sigma\sqrt{n}} \leq t, X_n = j, Z_n>0 \right) 
= \Phi^{+} (t) \bs \nu (j) u(i,z)
$$
and
$$
\lim_{n\to\infty} \bb P_{i,z} 
\left(  \frac{\log Z_n}{\sigma\sqrt{n}} \leq t, X_n = j \big| Z_n>0 \right) 
= \Phi^{+} (t) \bs \nu (j).
$$
\end{theorem}

As mentioned before, in the proofs of the stated results we make use of the previous developments in papers 
\cite{GrLauvLePage_2018-SPA, GLLP_CLLT_2017}.
These studies are based heavily on the existence of the harmonic function and the study of the 
asymptotic of the probability of the exit time for Markov chains which have been performed recently in 
\cite{grama_conditioned_2016, grama_limit_2016-1} and which are recalled in the next section.   
For recurrent Markov chains alternative approaches based on the renewal arguments are possible.   
The advantage of the harmonic function approach proposed here is that it could be extended 
for more general Markov environments which are not recurrent. In particular with these methods one could treat 
multi-type branching processes in random environments.   

The outline of the paper is as follows. 
In Section \ref{Sec: Preliminary results} we 
give a series of assertions for walks on Markov chains conditioned to stay positive
and prove Theorem \ref{Tsurvival001} for $z>1$. 
In Section \ref{sec:PrepBranching} we state some preparatory results for branching processes.
The proofs of Theorems \ref{T-Loi-limite001}, \ref{T-Loi-limite003} and \ref{T-Yaglom} 
are given in Sections \ref{sec:proofTh1} and \ref{sec: 44444}. 

We end this section by fixing some notation. 
As usual the symbol $c$ will denote positive constants depending on all previously introduced constants. In the same way 
the symbol $c$, enabled with subscripts, will denote positive constants depending only on the indices and all previously introduced constants. 
All these constants will change their values every occurrence. 
By $f \circ g$ we mean the composition of two function $f$ and $g$: $f \circ g (\cdot)=f(g(\cdot))$.
The indicator of an event $A$ is denoted by $\mathbbm 1_A$. 
For any bounded measurable function $f$ on $\bb X$, random variable $X$ in some measurable space $\bb X$ and event $A$, we set by definition  
$$\int_{\bb X}  f(x) \bb P (X \in \dd x, A) = \bb E\left( f(X); A\right) := \bb E \left(f(X) \mathbbm 1_A\right).$$

\section{Facts on Markov walks conditioned to stay positive} \label{Sec: Preliminary results}

\subsection{Conditioned limit theorems}
\label{Sec: MWCP}
We start by formulating two propositions which are consequences of the results in 
\cite{grama_limit_2016-1}, \cite{GrLauvLePage_2018-SPA} and \cite{GLLP_CLLT_2017}. 

Introduce the first time when the Markov walk $\left(y+ S_n \right)_{n\geq 0}$ becomes non-positive:  for any $y \in \bb R$, set
\begin{align}\label{def-tau001}
\tau_y := \inf \left\{ k \geq 1 : y+S_k \leq 0 \right\},
\end{align}
where $\inf \emptyset = 0.$
Conditions \ref{primitif}, \ref{Cond non-latice} and $\bs \nu(\rho) = 0$ ensure that
the stopping time $\tau_y$ is well defined and finite $\bb P_i$-almost surely, for any $i \in \bb X$. 

The following important proposition is a direct consequence of the results in \cite{grama_limit_2016-1} adapted to the case of a finite Markov chain. 
It proves the existence of the harmonic function related to Markov walk $\left(y+ S_n \right)_{n\geq 0}$
and states some of its properties to be used in the proofs of the main results of the paper.  
\begin{proposition} \label{PropF-hamon001}
Assume Conditions \ref{primitif}, \ref{Cond non-latice} and $k'(0) 
= 0$.
There exists a non-negative function $V$ on $\bb X \times \bb R$ such that
\begin{enumerate}[ref=\arabic*, leftmargin=*, label=\arabic*.]
	\item 
	For any $(i,y) \in \bb X \times \bb R$ and $n \geq 1$,
	\[
	\bb E_i \left( V \left( X_n, y+S_n \right) \,;\, \tau_y > n \right) = V(i,y).
	\]
	\item 
	For any $i\in \bb X$, the function $V(i,\cdot)$ is non-decreasing and for any $(i,y) \in \bb X \times \bb R$,
	\[
	V(i,y) \leq c \left( 1+\max(y,0) \right).
	\]
	\item 
	For any $i \in \bb X$, $y > 0$ and $\delta \in (0,1)$,
	\[
	\left( 1- \delta \right)y - c_{\delta} \leq V(i,y) \leq \left(1+\delta \right)y + c_{\delta}.
	\]
\end{enumerate}
\end{proposition}

We need the asymptotic of the probability of the event $\{ \tau_y > n\}$
jointly with the state of the Markov chain $(X_n)_{n\geq 1}$.

\begin{proposition} \label{Prop Init002}
Assume Conditions \ref{primitif}, \ref{Cond non-latice} and $k'(0)=0$.
\begin{enumerate}[ref=\arabic*, leftmargin=*, label=\arabic*.]
\item  
For any $(i,y) \in \bb X \times \bb R$ and $j \in \bb X$, we have
\begin{align*}
\lim_{n\to +\infty} \sqrt{n} \bb P_{i} \left( X_n = j \,,\, \tau_y > n \right) = \frac{2V(i,y) \bs \nu (j)}{\sqrt{2\pi} \sigma}. 
\end{align*}
\item For any $(i,y) \in \bb X \times \bb R$ and $n\geq 1$,
	\[
	\bb P_i \left( X_n=j,  \tau_y > n \right) \leq c\frac{ 1 + \max(y,0) }{\sqrt{n}}.
	\]
\end{enumerate}
\end{proposition}
For a proof of the first assertion of Proposition \ref{Prop Init002}, see Lemma 2.11 in \cite{GrLauvLePage_2018-SPA}.
The second is deduced from the point (b) of Theorem 2.3 of \cite{grama_limit_2016-1}.

Denote by $\supp(V) = \left\{ (i,y) \in \bb X \times \bb R : \right.$ $\left. V(i,y) > 0 \right\}$ the support of the function $V$. 
By the point 3 of Theorem \ref{PropF-hamon001}, the harmonic function $V$ satisfies the following property: 
for any $i\in \mathbb X$ there exist $y_i\geq 0$ such that $(i,y)\in \supp V$, for any $y>y_i$.

In addition to the previous two propositions we need the following result, 
which gives the asymptotic behaviour of the conditioned limit law
of the Markov walk $\left(y+ S_n \right)_{n\geq 0}$ jointly with the Markov chain  $\left(X_n \right)_{n\geq 0}$.
It extends Theorem 2.5 of \cite{grama_limit_2016-1} where the asymptotic of  $\frac{y+S_n}{\sigma \sqrt{n}}$ given the event 
$\{\tau_y >n\}$ has been considered.
\begin{proposition}
\label{Prop-CondLimTh}
Assume Conditions \ref{primitif}, \ref{Cond non-latice} and $k'(0)=0$. 
\begin{enumerate}[ref=\arabic*, leftmargin=*, label=\arabic*.]
\item \label{racine001} For any $(i,y) \in \supp(V)$ and $t\geq 0$,
\begin{align*}
	\bb P_i \left( \sachant{\frac{y+S_n}{\sigma \sqrt{n}} \leq t, X_n=j }{\tau_y >n} \right) \underset{n\to+\infty}{\longrightarrow} 
	 \Phi^+(t) \bs \nu (j).
\end{align*}
\item \label{racine002} There exists $\ee_0 >0$ such that, for any $\ee \in (0,\ee_0)$, $n\geq 1$, $t_0 > 0$, $t\in[0,t_0]$ and $(i,y) \in \bb X \times \bb R$,
\begin{align*}
	&\abs{ \bb P_i \left(\frac{y+S_n}{\sqrt{n} \sigma} \leq t, X_n=j, \tau_y > n \right) - \frac{2V(i,y)}{\sqrt{2\pi n}\sigma} \Phi^+(t) \bs \nu (j) } \\
	&\hskip6cm \leq c_{\ee,t_0} \frac{\left( 1+\max(y,0)^2 \right)}{n^{1/2+\ee}}.
\end{align*}
\end{enumerate}
\end{proposition}
\begin{proof}
It is enough to prove the point 2 of the proposition.
It will be derived from the corresponding result in Theorem 2.5 of \cite{grama_limit_2016-1}. 
We establish first an upper bound. 
Let $k=[n^{1/4}]$ and $\| \rho \|_{\infty}=\max_{i\in \mathbb X } | \rho(i) |$.  
Since 
$$
S_n=S_{n-k}+ \sum_{i=n-k+1}^{n} \rho(X_i),
$$
we have
\begin{align} \label{ProofPr23-001}
&\bb P_i \left(\frac{y+S_n}{\sqrt{n} \sigma} \leq t, X_n=j, \tau_y > n \right)  \nonumber \\
&\leq
\bb P_i \left(\frac{y+S_{n-k}}{\sqrt{n} \sigma} \leq t + \frac{k}{\sigma\sqrt{n}} \| \rho \|, X_n=j, \tau_y > n-k \right)  \nonumber \\
 &:=I (k,n).
\end{align}
By the Markov property
\begin{align*}
I (k,n)= \bb E_i \left( P^k(X_{n-k},j); \frac{y+S_{n-k}}{\sqrt{n} \sigma} \leq t + \frac{k}{\sigma\sqrt{n}} \| \rho \|, \tau_y > n-k \right). 
\end{align*}
Now using \eqref{EQ:exprate001} and setting $t_{n,k}=\frac{\sqrt{n}}{\sqrt{n-k}}(t + \frac{k}{\sigma\sqrt{n}}) \| \rho \|,$
\begin{align} \label{ProofPr23-002}
I(k,n) \leq  (\bs \nu(j) + c_1\e^{-c_2 k}) 
\bb P_i \left(\frac{y+S_{n-k}}{\sqrt{n-k} \sigma} \leq t_{n,k}, \tau_y > n-k \right). 
\end{align}
By Theorem 2.5 and Remark 2.10 of \cite{grama_limit_2016-1}, 
there exists $\varepsilon_0>0$ such that for any $\varepsilon \in (0, \varepsilon_0)$
and $t_0>0$, $n\geq1$, we have, for $t_{n,k}\leq t_0 $, 
\begin{align} \label{ProofPr23-003}
&\bb P_i \left(\frac{y+S_{n-k}}{\sqrt{n-k} \sigma} \leq t_{n,k}, \tau_y > n-k \right) \nonumber  \\
&\leq \frac{2V(i,y)}{\sqrt{2\pi (n-k)}\sigma} \Phi^+(t_{n,k}) 
+ c_{\varepsilon, t_0}\frac{(1+\max\{0,y\})^2   }{ (n-k)^{1/2+\varepsilon/16}}.
\end{align}
Since $| t_{n,k} - t| \leq c_{t_0}\frac{1}{n^{1/4}}$  and $\Phi^+$ is smooth, 
we obtain
\begin{align} \label{ProofPr23-004}
\frac{2V(i,y)}{\sqrt{2\pi (n-k)}\sigma} \Phi^+(t_{n,k}) 
\leq 
\frac{2V(i,y)}{\sqrt{2\pi n}\sigma} \Phi^+(t) + c_{t_0}\frac{(1+\max\{0,y\})^2}{n^{1/2 +1/4}}.
\end{align}
From \eqref{ProofPr23-001}, \eqref{ProofPr23-002}, \eqref{ProofPr23-003} and \eqref{ProofPr23-004} it follows that
\begin{align} \label{ProofPr23-005}
\bb P_i &\left(\frac{y+S_n}{\sqrt{n} \sigma} \leq t, X_n=j, \tau_y > n \right) \nonumber \\
&\leq
 \bs \nu(j) \frac{2V(i,y)}{\sqrt{2\pi n}\sigma} \Phi^+(t) 
+ c_{\varepsilon, t_0}\frac{(1+\max\{0,y\})^2   }{ n^{1/2+\varepsilon/16}}.
\end{align}

Now we shall establish a lower bound. 
With the notation introduced above, we have
\begin{align} \label{ProofPr23-201}
&\bb P_i \left(\frac{y+S_n}{\sqrt{n} \sigma} \leq t, X_n=j, \tau_y > n \right) \nonumber \\
&\geq \bb P_i \left(\frac{y+S_{n-k}}{\sqrt{n} \sigma} \leq t - \frac{k}{\sigma\sqrt{n}} \| \rho \|, X_n=j, \tau_y > n-k \right) \nonumber \\
& -\bb P_i \left( 
n-k < \tau_y \leq n \right) \nonumber  \\
&:=I_1 (k,n) - I_2(k,n). 
\end{align}
As in the proof of \eqref{ProofPr23-005}, we establish the lower bound
\begin{align} \label{ProofPr23-202}
 I_1(k,n) 
\geq
 \bs \nu(j) \frac{2V(i,y)}{\sqrt{2\pi n}\sigma} \Phi^+(t) 
- c_{\varepsilon, t_0}\frac{(1+\max\{0,y\})^2   }{ n^{1/2+\varepsilon/16}}.
\end{align}
Note that $0 \geq  \min_{n-k  < i \leq n}  \{  y+S_i  \} \geq y+S_{n-k} - k \|  \rho \|_{\infty}$, on the set 
$\{ n-k < \tau_y \leq n \}$. Set $t_{n,k}=\frac{k \|  \rho \|_{\infty} }{\sigma\sqrt{n-k}}$. Then, 
\begin{align} \label{ProofPr23-203}
I_2(k,n) &= \bb P_i \left( n-k < \tau_y \leq n \right) \nonumber\\
&\leq \bb P_i \left( \frac{y+S_{n-k}}{\sigma\sqrt{n-k}} \leq t_{n,k}, 
n-k < \tau_y \leq n \right) \nonumber\\
&\leq \bb P_i \left( \frac{y+S_{n-k}}{\sigma\sqrt{n-k}} \leq t_{n,k}, 
\tau_y \geq n -k \right).
\end{align}
Again by Theorem 2.5 and Remark 2.10 of \cite{grama_limit_2016-1}, 
there exists $\varepsilon_0>0$ such that for any $\varepsilon \in (0, \varepsilon_0)$
and $t_0>0$, $n\geq1$, we have, for $t_{n,k}\leq t_0 $, 
\begin{align} \label{ProofPr23-204}
&\bb P_i \left(\frac{y+S_{n-k}}{\sqrt{n-k} \sigma} \leq t_{n,k}, \tau_y > n-k \right) \nonumber  \\
&\leq \frac{2V(i,y)}{\sqrt{2\pi (n-k)}\sigma} \Phi^+(t_{n,k}) 
+ c_{\varepsilon, t_0}\frac{(1+\max\{0,y\})^2   }{ (n-k)^{1/2+\varepsilon/16}}.
\end{align}
Since $| t_{n,k} | \leq c_{t_0}\frac{1}{n^{1/4}}$ and $\Phi^+(0)=0$, we obtain
\begin{align} \label{ProofPr23-205}
\frac{2V(i,y)}{\sqrt{2\pi (n-k)}\sigma} \Phi^+(t_{n,k}) 
&\leq 
\frac{2V(i,y)}{\sqrt{2\pi n}\sigma} \Phi^+(0) + c_{t_0}\frac{(1+\max\{0,y\})^2}{n^{1/2 +1/4}} \nonumber  \\
&=c_{t_0}\frac{(1+\max\{0,y\})^2}{n^{1/2 +1/4}}.
\end{align}
From \eqref{ProofPr23-203}, \eqref{ProofPr23-204} and \eqref{ProofPr23-205}, we deduce that 
\begin{align}\label{ProofPr23-206}
I_2(k,n) \leq c_{\varepsilon, t_0}\frac{(1+\max\{0,y\})^2   }{ n^{1/2+\varepsilon/16}}.
\end{align}
Using \eqref{ProofPr23-201}, \eqref{ProofPr23-202} and \eqref{ProofPr23-206}, one gets
\begin{align*} 
\bb P_i &\left(\frac{y+S_n}{\sqrt{n} \sigma} \leq t, X_n=j, \tau_y > n \right) \\
&\geq
 \bs \nu(j) \frac{2V(i,y)}{\sqrt{2\pi n}\sigma} \Phi^+(t) 
- c_{\varepsilon, t_0}\frac{(1+\max\{0,y\})^2   }{ n^{1/2+\varepsilon/16}},
\end{align*}
which together with \eqref{ProofPr23-005} end the proof of the point 2 of the proposition. The point 1 follows from the point 2.
\end{proof}

We need the following estimation, 
whose proof can be found in 
\cite{GLLP_CLLT_2017}. 
  
\begin{proposition} \label{LLTC}
Assume Conditions \ref{primitif}, \ref{Cond non-latice} and $k'(0)=0$. 
Then there exists $c > 0$ such that for any $a>0$, non-negative function $\psi \in \scr C(\mathbb X)$, $y \in \bb R$, $t \geq 0$ and $n \geq 1$,
\begin{align*}
\sup_{i\in \bb X} \bb E_{i} &\left( \psi \left( X_{n} \right) \,;\, y+S_{n} \in [t,t+a] \,,\, \tau_y > n \right) \\
&\leq \frac{c \norm{\psi}_{\infty}}{n^{3/2}} \left( 1+a^3 \right)\left( 1+t \right)\left( 1+\max(y,0) \right).
\end{align*}
\end{proposition}

\subsection{Change of probability measure}

Fix any $(i,y) \in \supp(V)$ and $z\in \mathbb N$. 
The harmonic function $V$ from Proposition \ref{PropF-hamon001}, 
allows us to introduce the probability measure $\bb P_{i,y,z}^+$ on 
$\bb (\mathbb X \times 
\mathbb N )^{\mathbb N}$
and the corresponding expectation $\bb E_{i,y,z}^+$, by the following relation: 
for any $n\geq 1$ and any bounded measurable $g$: $\bb (\mathbb X  \times \mathbb N )^{n} \mapsto \bb R$,
\begin{align}	\label{changemes001}
	&\bb E_{i,y,z}^+ \left( g \left( X_1,Z_1, \dots, X_n,Z_n \right) \right) \nonumber\\
	&:= \frac{1}{V(i,y)} \bb E_{i,z} \big( g\left( X_1,Z_1, \dots, X_n,Z_n \right) \nonumber\\
	&\qquad\qquad\qquad \times V\left( X_n, y+S_n \right) \,;\, \tau_y > n \big).
\end{align}
The fact that the function $V$ is harmonic (by point 1 of Proposition \ref{PropF-hamon001})
ensures the applicability of the Kolmogorov extension theorem and shows that $\bb P_{i,z,y}^+$  
is a probability measure. 
In the same way we define the probability measure $\bb P_{i,y}^+$ 
and the corresponding expectation $\bb E_{i,y}^+$:
for any $(i,y) \in \supp(V)$, $n \geq 1$ and any bounded measurable $g$: $\bb X^n \to \bb R$,
\begin{align} \label{changemes002}
&\bb E_{i,y}^+ \left( g \left( X_1, \dots, X_n \right) \right) \nonumber\\
&:= \frac{1}{V(i,y)} \bb E_{i} \left( g\left( X_1, \dots, X_n \right) V\left( X_n, y+S_n \right) \,;\, \tau_y > n \right).
\end{align}
The relation between the expectations $\bb E_{i,y,z}^+$
and $\bb E_{i,y}^+$ is given by the following identity: 
for any  $n \geq 1$ and 
any bounded measurable $g$: $\bb X^n \to \bb R$, 
\begin{align}	\label{changemes000}
	\bb E_{i,y,z}^+ \left( g \left( X_1, \dots, X_n \right) \right) = \bb E_{i,y}^+ \left( g \left( X_1, \dots, X_n \right) \right). 
\end{align}

With the help of Proposition \ref{LLTC}, we have the following bounds. 
\begin{lemma} \label{lemma sum eS_n }
Assume Conditions \ref{primitif}, \ref{Cond non-latice} and $k'(0)=0$. 
For any $(i,y) \in \supp(V)$, we have, for any $k \geq 1$,
\[
\bb E_{i,y}^+ \left( \e^{-S_k}\right) \leq \frac{c \left( 1+\max(y,0) \right)\e^{y}}{k^{3/2} V(i,y)}.
\]
In particular,
\[
\bb E_{i,y}^+ \left( \sum_{k=0}^{+\infty} \e^{-S_k} \right) \leq \frac{c \left( 1+\max(y,0) \right)\e^{y}}{V(i,y)}.
\]
\end{lemma}
The proof being similar to that in \cite{GrLauvLePage_2018-SPA} is left to the reader.


We need the following statements. 
Let
$
\mathscr F_n = \sigma \{ X_0,Z_0, \ldots, X_n, Z_n  \}
$
and 
$(Y_n)_{n\geq 0}$ a  bounded $(\mathscr F_n)_{n\geq 0}$-adapted sequence.
\begin{lemma} \label{lemmaE+}
Assume Conditions \ref{primitif}-\ref{Cond non-latice} and $k'(0)=0$. 
For any $k \geq 1$,
$(i,y) \in \supp(V)$, $z\in \mathbb N$ 
and $j \in \bb X$,
\[
\lim_{n\to +\infty} \bb E_{i,z} \left( \sachant{ 
Y_k 
\,;\, X_n = j}{ \tau_y > n } \right) 
= \bb E_{i,y,z}^+ \left( Y_k 
\right) \bs \nu (j).
\]
\end{lemma}

\begin{proof}
For the sake of brevity, for any $(i,j) \in \bb X^2$, $y \in \bb R$ and $n \geq 1$, set
\[
P_n(i,y,j) := \bb P_i \left( X_n = j \,,\, \tau_y > n \right).
\]
Fix $k \geq 1$. 
By the point 1 of Proposition \ref{Prop Init002}, it is clear that for any $(i,y) \in \supp(V)$ and $n$ large enough, $\bb P_i \left( \tau_y > n \right) > 0$. 
By the Markov property, for any $j \in \bb X$ 
and $n \geq k+1$ large enough,
\begin{align*}
I_0 &:= \bb E_{i,z} \left( \sachant{ Y_k \,;\, X_n = j}{ \tau_y > n } \right)\\ 
&= \bb E_{i,z} \left( Y_k \frac{P_{n-k} \left( X_k,y+S_k, j \right)}{\bb P_i \left( \tau_y > n \right)} \,;\, \tau_y > m \right).
\end{align*}
Using the point 1 of Proposition \ref{Prop Init002}, by the Lebesgue dominated convergence theorem,
\begin{align*}
	\lim_{n\to+\infty} I_0 &= 
	\bb E_{i,z} \left( Y_k  \frac{V \left( X_k,y+S_m \right)}{V(i,y)} \,;\, \tau_y > k \right) \bs \nu(j) \\
	&= \bb E_{i,y,z}^+ \left( Y_k \right) \bs \nu (j).
\end{align*}
\end{proof}

\begin{lemma}\label{lemma limit 002}
Assume that $(i,y) \in \supp V $ and $z\in \mathbb N$.
For any  bounded $(\mathscr F_n)_{n\geq 0}$-adapted sequence $(Y_n)_{n\geq 0}$
such that $Y_n\to Y_{\infty}$ $\mathbb P_{i,y,z}^{+}$-a.s.,
\begin{align*} 
\limsup_{k\to \infty} \limsup_{n\to \infty} 
\sqrt{n}  
\mathbb E_{i,z} \big( \big|  Y_n-Y_k \big|;  \tau_y > n \big) =0.
\end{align*}
\end{lemma}
\begin{proof}
Let $k\geq1$ and $\theta>1$. 
Then
\begin{align} \label{eqPPPPP001} 
\mathbb E_{i,z} \big( \big|   Y_n-Y_k \big|;   \tau_y > n \big)   
&= \mathbb E_{i,z} \big(   \big|Y_n-Y_k \big|;   \tau_y > \theta n \big) \nonumber  \\ 
& + \mathbb E_{i,z} \big(   \big|Y_n-Y_k \big|;  n< \tau_y \leq \theta n \big).  
\end{align}
We bound the second therm in the right-hand side of \eqref{eqPPPPP001}:
\begin{align} \label{eqPPPPP002} 
&\mathbb E_{i,z} \big(   \big|Y_n-Y_k \big|;  n< \tau_y \leq \theta n \big)
\leq 
C \mathbb P_{i,z} \big(    n< \tau_y \leq \theta n \big)  
\end{align}
By  point 1 of Proposition \ref{Prop Init002}, we have 
\begin{align}\label{eqPPPPP002a}
&\lim_{n\to \infty} \sqrt{n} \mathbb P_{i,z} \big(    n< \tau_y \leq \theta n \big) \nonumber \\
&= \lim_{n\to \infty} \sqrt{n} \mathbb P_{i,z} \big(    \tau_y >n \big)
- \lim_{n\to \infty} \sqrt{n} \mathbb P_{i,z} \big(   \tau_y > \theta n \big) \nonumber\\
&= \frac{2V(i,y)}{\sqrt{2\pi}\sigma} \Big(1-\frac{1}{\sqrt{\theta}}\Big).
\end{align}
Now we shall prove that
\begin{align}\label{eqBBBB001}
\limsup_{k\to \infty} \limsup_{n\to \infty}  \sqrt{n} \mathbb E_{i,z} \big(   \big|Y_n-Y_k \big|;   \tau_y > \theta n \big) =0.
\end{align}
Recall that $\theta >1$. By the Markov property (conditioning on $\mathscr F_n$), 
\begin{align} \label{eqPPPPP003} 
&\abs{\mathbb E_{i,z} \big(   \big|Y_n-Y_k \big|;   \tau_y > \theta n \big)} \nonumber \\ 
&\qquad = \mathbb E_{i,z} \big(   \abs{Y_n-Y_k} P_{[(\theta-1)n]}(X_n,y+S_n) ;  \tau_y > n \big), 
\end{align}
where we use the notation $P_{n'}(i',y') :=  \mathbb P_{i',y'} \big( \tau_y' > n' \big).$ 
By point 2 of Proposition \ref{Prop Init002} 
and by point 3 of Proposition \ref{PropF-hamon001}, 
there exists $y_0>0$ such that for $i' \in \mathbb X$, $y' > y_0$ and $n'\in\mathbb N$,
\begin{align}\label{eqBBBB002}
P_{n'}(i',y') 
\leq  \frac{c}{\sqrt{n'}} \left(1+\max\{0,y'\} \right) \leq c\frac{V(i',y')}{\sqrt{n'}}.
\end{align}
Representing the right-hand side of \eqref{eqPPPPP003} 
as a sum of two terms and using \eqref{eqBBBB002}, gives
\begin{align} \label{eqPPPPP003b} 
&\sqrt{n}\abs{\mathbb E_{i,z} \big(   \big|Y_n-Y_k \big|;   \tau_y > \theta n \big)} \nonumber \\ 
& = \sqrt{n}\mathbb E_{i,z} \big(   \abs{Y_n-Y_k} P_{[(\theta-1)n]}(X_n,y+S_n) ; y+S_n  \leq y_0,  \tau_y > n \big) \nonumber  \\ 
& \quad +  \sqrt{n}\mathbb E_{i,z} \big(   \abs{Y_n-Y_k} P_{[(\theta-1)n]}(X_n,y+S_n) ; y+S_n  > y_0,  \tau_y > n \big) \nonumber  \\ 
&\leq 
c \sqrt{n}\mathbb P_{i} \big( y+S_n  \leq y_0, \tau_y > n \big) \nonumber  \\
&+ \frac{c}{\sqrt{\theta-1}} 
\mathbb E_{i,z} \big(   \abs{Y_n-Y_k} V(X_n, y+S_n);  \tau_y > n \big).
\end{align}
Using point 1 of Proposition \ref{Prop-CondLimTh} and point 1 of Proposition \ref{Prop Init002}, we have  
\begin{align} \label{eqPPPPP004}
\lim_{n\to\infty}\sqrt{n}\mathbb P_{i} \big( y+S_n  \leq y_0, \tau_y > n \big)=0.
\end{align}
By the change of measure formula \eqref{changemes001},
\begin{align} \label{eqPPPPP005}
&\mathbb E_{i,z} \big(   \abs{Y_n-Y_k} V(X_n, y+S_n);   \tau_y > n \big) \nonumber  \\
&\qquad\qquad = V(i,y) \mathbb E_{i,y,z}^{+} \big(   \abs{Y_n-Y_k} \big).
\end{align}
Letting first $n\to\infty$ and then $k\to\infty$, by the Lebesgue dominated convergence theorem,
\begin{align}\label{eqPPPPP006}
\limsup_{k\to \infty} \limsup_{n\to \infty} \mathbb E_{i,y,z}^{+} \big(   \abs{Y_n-Y_k} \big) =0.
\end{align}
From \eqref{eqPPPPP003b}-\eqref{eqPPPPP006} 
we deduce \eqref{eqBBBB001}.
Now \eqref{eqPPPPP001}-\eqref{eqBBBB001} imply that, for any $\theta>1$,
\begin{align*}
\limsup_{k\to \infty} \limsup_{n\to \infty} \sqrt{n} \mathbb E_{i,z} \big( \big|   Y_n-Y_k \big|;   \tau_y > n \big)  
\leq \frac{2V(i,y)}{\sqrt{2\pi}\sigma} \Big(1-\frac{1}{\sqrt{\theta}}\Big).
\end{align*}
Since $\theta $ can be taken arbitrarily close to $1$, we conclude the claim of the lemma. 
\end{proof}

The next assertion is an easy consequence of Lemmata \ref{lemmaE+} and \ref{lemma limit 002}.
\begin{lemma}\label{lemma limit 003}
Assume that $(i,y) \in \supp V $ and $z\in \mathbb N$. For any  bounded $(\mathscr F_n)_{n\geq 0}$-adapted sequence $(Y_n)_{n\geq 0}$
such that $Y_n\to Y_{\infty}$ $\mathbb P_{i,y,z}^{+}$-a.s.,
\begin{align*} 
\lim_{n\to +\infty}  \mathbb E_{i,z} \big(   Y_n;  X_n=j \big|  \tau_y > n \big) 
= \mathbb E_{i,y,z}^{+} \big(   Y_{\infty}\big) \bs \nu (j) .
\end{align*}
\end{lemma}
\begin{proof}
For any  $n > k \geq1$, we have
\begin{align} \label{FFF001} 
&\lim_{n\to \infty}  \sqrt{n}  \mathbb E_{i,z} \big(   Y_n;  X_n=j,  \tau_y > n \big) \nonumber \\ 
&= \sqrt{n}  \mathbb E_{i,z} \big(   Y_k;  X_n=j,  \tau_y > n \big) 
+ \sqrt{n}  \mathbb E_{i,z} \big(   Y_n-Y_k;  X_n=j,  \tau_y > n \big)  .
\end{align}
By Lemma \ref{lemmaE+}, the first term in the right-hand side of \eqref{FFF001} converges to
$\frac{2V(i,y)}{\sqrt{2\pi}\sigma} \bs \nu (j) \mathbb E_{i,y,z}^{+} Y_{k}$ 
as $n\to\infty$, where 
$ \lim_{k\to \infty} \mathbb E_{i,z,y}^{+} Y_{k} = \mathbb E_{i,y,z}^{+} Y_{\infty}.$ 
By Lemma \ref{lemma limit 002}, the second term in the r.h.s.\ of \eqref{FFF001} vanishes, which completes the proof. 
\end{proof}
%

\subsection{The dual Markov chain} 

Note that the invariant measure $\bs \nu$ is positive on $\bb X$. Therefore 
the dual Markov kernel
\begin{equation}
\label{statueBP} 
\bf P^* \left( i,j \right) = \frac{\bs \nu \left( j \right)}{\bs \nu (i)} \bf P \left( j,i \right), \quad  i,j \in \bb X
\end{equation} 
is well defined.
On an extension of the probability space $(\Omega, \scr F, \bb P)$ we consider
the dual Markov chain $\left( X_n^* \right)_{n\geq 0}$ with values in $\bb X$ and with transition probability $\bf P^*$.
The dual chain $\left( X_n^* \right)_{n\geq 0}$ can be chosen to be 
independent of the chain $\left( X_n \right)_{n\geq 0}$.
Accordingly, the dual Markov walk $\left(S_n^* \right)_{n\geq 0}$ is defined by setting
\begin{equation}
\label{promenade001}
S_0^* = 0 \qquad \text{and} \qquad S_n^* = -\sum_{k=1}^n \rho \left( X_k^* \right), \quad  n \geq 1.
\end{equation}
For any $y \in \bb R$ define the first time when the Markov walk $\left(y+ S_n^* \right)_{n\geq 0}$ becomes non-positive:
\begin{equation}
\label{timetau001}
\tau_y^* := \inf \left\{ k \geq 1 : y+S_k^* \leq 0 \right\}.
\end{equation}
For any $i\in \bb X$, denote by $\bb P_i^*$ and $\bb E_i^*$ the probability and the associated expectation generated by the finite dimensional distributions of the Markov chain $( X_n^* )_{n\geq 0}$ starting at $X_0^* = i$.

It is easy to verify (see \cite{GrLauvLePage_2018-SPA}), that $\bs \nu$ is also $\mathbf P^*$ invariant and that 
Conditions \ref{primitif} and \ref{Cond non-latice} are satisfied for $\mathbf P^*$. 
This implies that Propositions \ref{PropF-hamon001}-\ref{LLTC} formulated in Subsection \ref{Sec: MWCP} hold also for the 
dual Markov chain $(X_n^*)_{n\geq 0}$ 
and the Markov walk $\left(y+ S_n^* \right)_{n\geq 0}$,
with the harmonic function $V^*$ such that,
for any $(i,y) \in \bb X \times \bb R$ and $n \geq 1$,
\begin{align*}
\bb E_i \left( V^* \left( X_n, y+S_n^* \right) \,;\, \tau^*_y > n \right) = V^*(i,y).
\end{align*}

The following duality property is obvious (see \cite{GrLauvLePage_2018-SPA}):
\begin{lemma}[Duality]
\label{dualityBP}
For any $n\geq 1$ and any function $g$: $\bb X^n \to \bb C$,
\[
\bb E_i \left( g \left( X_1, \dots, X_n \right) \,;\, X_{n+1} = j \right) = \bb E_j^* \left( g \left( X_n^*, \dots, X_1^* \right) \,;\, X_{n+1}^* = i \right) \frac{\bs \nu(j)}{\bs \nu(i)}.
\]
\end{lemma}


\section{Preparatory results for branching processes}\label{sec:PrepBranching}
All over the remaining part of the paper we will use the following notation.
For  $ s \in [0,1)$, let
\begin{align*} 
\varphi_{X_k}(s) =\frac{1}{1-f_{X_k}(s)} - \frac{1}{f'_{X_k}(1)(1-s)},
\end{align*}
and, by continuity,
\begin{align*} 
\varphi_{X_k}(1) = \lim_{s\to 1} \varphi_{X_k}(s)  = \frac{ f''_{X_k}(1) }{2 f'_{X_k}(1)^2}.
\end{align*}
In addition,   
for $ s \in [0,1)$, let  $g_{z}(s)=s^z$ and
\begin{align*} 
\psi_{z}(s) =\frac{1}{1-g_{z}(s)} - \frac{1}{g'_{z}(1)(1-s)} = \frac{1}{1-s^z} - \frac{1}{z(1-s)} ,
\end{align*}
and, by continuity,
\begin{align*} 
\psi_{z}(1) = \lim_{s\to 1} \psi_{z}(s)  = \frac{ z(z-1) }{2 z^2}= \frac{1}{2}\frac{ z-1 }{z}.
\end{align*}

For any $n\geq 1$, $z\geq 1$ and $s\in [0,1]$, 
the following quantity will play an important role in our study:
\begin{align*} 
q_{n,z} (s) =  1-\big( f_{X_1}\circ \cdots \circ f_{X_n} (s)\big)^z.
\end{align*}
Under Condition \ref{Cond moments}, for any $i\in \mathbb X$ and $s\in [0,1]$ we have $f_i(s)\in [0,1]$
and $f_{X_1}\circ \cdots \circ f_{X_n} (s) \in [0,1]$. 
This implies that, for any $s\in [0,1]$, 
\begin{align} \label{bound for q_{n,z}(s)}
q_{n,z} (s) \in [0,1].
\end{align}
For any $n\geq 1$ and $z\geq 1$, the function $s\mapsto q_{n,z} (s) $ is convex on $[0,1]$.   
Since the sequence $( \xi_i^{n,j} )_{j,n \geq 1}$ is independent 
of the Markov chain $\left( X_n \right)_{n\geq 0},$ with $s=0$, 
we have, $\mathbb P_{i,y,z}^{+}$-a.s., 
\begin{align}\label{SURV000-in0}
q_{n,z} (0) 
= \mathbb P_{i,y,z}^{+} ( Z_n >0  \big| (X_k)_{k\geq 0} ). 
\end{align}
Note also that  $\{ Z_{n} >0\} \supset \{ Z_{n+1} >0\}$
and therefore, for $n\geq 1$,
\begin{align} \label{q_nz inequality}
q_{n,z} (0) \geq q_{n+1,z} (0).
\end{align}
Taking the limit as $n\to\infty$, $\mathbb P_{i,y,z}^{+}$-a.s.,
\begin{align}\label{SURV000}
\lim_{n\to \infty}  q_{n,z} (0) 
&= \lim_{n\to \infty}  \mathbb P_{i,y,z}^{+} (Z_n >0 \big| (X_k)_{k\geq 0}  ) \nonumber \\
&=  \mathbb P_{i,y,z}^{+} ( \cap_{n\geq 1} \{ Z_n >0\} \big| (X_k)_{k\geq 0} ). 
\end{align}
Moreover, by convexity of the function $q_{n,z}(s)$ we have 
\begin{align}\label{Ineqz001}
q_{n,z}(0) \leq z e^{S_n}.
\end{align}

The following formula (whose proof is left to the reader) is similar to the well-known statements from the papers by Agresti \cite{agresti_bounds_1974} 
and Geiger and Kersting \cite{geiger_survival_2001}: 
for any $s \in [0,1)$ and $n \geq 1$, 
\begin{align} \label{Agresti000}
\frac{1}{q_{n,z} (s)} 
&= \frac{1}{z f'_{X_1}(1)  \cdots  f'_{X_n}(1)   (1-s) } \nonumber \\
&+ \frac{1}{z}\sum_{k=1}^{n} \frac{\varphi_{X_{k}} \circ f_{X_{k+1}} \circ \cdots  \circ f_{X_n}   (s)}{f'_{X_1}(1)  \cdots  f'_{X_{k-1}}(1)  }  \nonumber\\
&+ \psi_{z} \circ f_{X_{1}} \circ \cdots  \circ f_{X_n}  (s).
\end{align}
We can rewrite \eqref{Agresti000} 
in the following more convenient form: 
for any $s \in [0,1)$ and $n \geq 1$, 
\begin{align}\label{AAA01}
q_{n,z} (s)^{-1} &= \frac{1}{z}\bigg( \frac{\e^{-S_n}}{1-s} + \sum_{k=0}^{n-1} \e^{-S_{k}} \eta_{k+1,n}(s) \bigg) \nonumber\\
	&+ \psi_{z} \circ f_{X_{1}} \circ \cdots  \circ f_{X_n}  (s),
\end{align}    
where 
$$
\eta_{k,n}(s) =  \varphi_{X_{k}} \circ f_{X_{k+1}} \circ \cdots  \circ f_{X_n}   (s).
$$
Since $\frac{1}{2}\varphi(0) \leq  \varphi(s) \leq 2 \varphi(1) $, for any $k \in \{ 1, \dots, m \}$, 
\begin{equation}\label{AAA02}
0 \leq \eta_{k,m}(s) \leq  \frac{f''_{X_{k}}(1)}{f'_{X_{k}}(1)^2} \leq \eta:=\max_{i\in \mathbb X} \frac{f''_{i}(1)}{f'_i (1)^2}.
\end{equation}
By Theorem 5 of \cite{athreya1971branching1}, for any $(i,y) \in \supp(V)$, $s \in[0,1)$, $m\geq 1$ and $k \in \{ 1, \dots, m \}$, 
there exists a random variable $\eta_{k,\infty}(s)$, such that
\begin{align} \label{AAA02a}
\lim_{n\to+\infty} \eta_{k,n}(s) = \eta_{k,\infty}(s) 
\end{align}
everywhere, and by \eqref{AAA02}, for any $s \in[0,1)$ and $k \geq 1 $,
\begin{equation}
	\label{AAA02b}
	\eta_{k,\infty}(s) \in [0,\eta].
\end{equation}
In the same way, 
\begin{align} \label{AAA501}
\lim_{n\to+\infty}  \psi_{z} \circ f_{X_{1}} \circ \cdots  \circ f_{X_n}  (s) = \psi_{z,\infty}(s) \in [0,\frac{z-1}{2}] 
\end{align}
everywhere.
For any $s \in [0,1)$, define $q_{\infty,z}(s)$ by setting 
\begin{equation}
	\label{AAA02c}
	q_{\infty,z}(s)^{-1} := \frac{1}{z}\left[ \sum_{k=0}^{+\infty} \e^{-S_k} \eta_{k+1,\infty}(s) \right] + \psi_{z,\infty}(s).
\end{equation}
By Lemma \ref{lemma sum eS_n }, we have that 
\begin{align} \label{AAA02d}
\bb E_{i,y}^+ q_{\infty,z}(s)^{-1} <+\infty.
\end{align}

\begin{lemma} \label{lemma q_n conv vers q }
Assume Conditions \ref{primitif}, \ref{Cond non-latice} and $k'(0)=0$. 
For any $(i,y) \in \supp V$, $z\geq 1$ and $s\in [0,1)$,
\begin{align*}
\lim_{n\to+\infty}  \bb E_{i,y}^+ \left|\frac{1}{q_{n,z}(s) } - \frac{1}{q_{\infty,z}(s) } \right| = 0
\end{align*}
and
\begin{align*}
\lim_{n\to+\infty}  \bb E_{i,y}^+ \left| q_{n,z}(s) - q_{\infty,z}(s) \right| = 0.
\end{align*}
\end{lemma}

\begin{proof}
We give a sketch only.
Following the proof of Lemma 3.2 in \cite{GrLauvLePage_2018-SPA}
for any $(i,y) \in \supp V$,  by \eqref{AAA01}, \eqref{AAA02} and \eqref{AAA02b}, we obtain
\begin{align} \label{AAA606}
\bb E_{i,y}^+ \left( \abs{q_{n,z}^{-1}(s) - q_{\infty,z}^{-1}(s)} \right) 
&\leq \frac{1}{z(1-s)}\bb E_{i,y}^+ \left( \e^{-S_n} \right) \nonumber\\ 
&+ \frac{1}{z} \bb E_{i,y}^+ \left( \sum_{k=0}^{l} \e^{-S_k} \abs{\eta_{k+1,n}(s) - \eta_{k+1,\infty}(s) } \right) \nonumber\\
&+ \frac{2\eta}{z} \bb E_{i,y}^+ \left( \sum_{k=l+1}^{+\infty} \e^{-S_k} \right) \nonumber\\
	&+ \bb E_{i,y}^+ \big|\psi_{z}\circ f_{X_1} \circ\dots \circ f_{X_n}(s) - \psi_{z,\infty}(s)\big|.
\end{align}
The last term in the right hand side of \eqref{AAA606} converges to $0$ as $n\to \infty$ by \eqref{AAA501}.  
By Lemma \ref{lemma sum eS_n }  and the Lebesgue dominated convergence theorem, we have
\begin{align*}
	\limsup_{n\to\infty} \bb E_{i,y}^+ \left( \abs{q_{n,z}^{-1}(s) - q_{\infty,z}^{-1}(s)} \right) 
	\leq \frac{2\eta}{z} \bb E_{i,y}^+ \left( \sum_{k=l+1}^{+\infty} \e^{-S_k} \right).
\end{align*}
Taking the limit as $l\to \infty$, again by Lemma \ref{lemma sum eS_n }, we conclude the first assertion of the lemma. 
The second assertion follows from the first one, since $q_{n,z}(s) \leq 1$ and $q_{\infty,z}(s)\leq 1$. 
\end{proof}

\begin{lemma} \label{lemma boundprob001}
Assume Conditions \ref{primitif}, \ref{Cond non-latice} and $k'(0)=0$.
For any $(i,y)\in \supp V$ and $z\in \mathbb N$, $z\not=0$,
we have, for any $k \geq 1$, $\mathbb P_{i,y,z}^{+}$-a.s.,
\begin{align*}
\mathbb P_{i,y,z}^{+} (  \cup_{k\geq 1} \{ Z_k =0\} \big| (X_k)_{k\geq 1} ) <1.
\end{align*}
\end{lemma}
\begin{proof} 
By \eqref{SURV000} we have, $\mathbb P_{i,z,y}^{+}$-a.s., 
\begin{align*}
1- \mathbb P_{i,z,y}^{+} (  \cup_{k\geq 1} \{ Z_k =0\} \big| (X_k)_{k\geq 1} )
= \lim_{n\to \infty}  q_{n,z} (0).
\end{align*}
Using \eqref{AAA01} and  \eqref{AAA02} 
\begin{align}\label{AAA04}
\mathbb E_{i,y}^{+} q_{n,z} (0)^{-1} &\leq  \frac{1}{z} \mathbb E_{i,y}^{+} \bigg( \e^{-S_n} + \eta \sum_{k=0}^{n}  \e^{-S_{k-1}}  \bigg) + 1.
\end{align}    
By Lemma \ref{lemma sum eS_n } and a monotone convergence argument,
\begin{align}\label{AAA05}
\mathbb E_{i,y}^{+} \lim_{n\to\infty} q_{n,z} (0)^{-1} = \lim_{n\to\infty}\mathbb E_{i,y}^{+} q_{n,z} (0)^{-1} < \infty.
\end{align}    
Thus, $\mathbb P_{i,y}^{+}$-a.s.,
\begin{align*}
\lim_{n\to\infty} q_{n,z} (0)  >0,
\end{align*}
which ends the proof of the lemma. 
\end{proof}

We will make use of the following lemma:
\begin{lemma}\label{lemma bound tau<=n}
There exists a constant $c$ such that, for any $z\in \mathbb N$, $z\not=0,$
and $y\geq 0$ sufficiently large,
\begin{align*} 
\sup_{i\in \mathbb X}\mathbb P_{i,z} \bigg(  Z_n >0, \tau_{y} \leq n \bigg) 
\leq  c z\frac{ e^{-y} (1+ \max\{y,0 \})}{\sqrt{n}}.
\end{align*}
\end{lemma}
\begin{proof}
We follow the same line as the proof of Theorem 1.1 in \cite{GrLauvLePage_2018-SPA}. First, we have 
\begin{align*}
\mathbb P_{i,z}(Z_n>0, \tau_y \leq n) = \mathbb P_i(q_{n,z}(0); \tau_y \leq n).
\end{align*}
Using \eqref{Ineqz001}
and the fact that 
$q_{n,z}(0)$ is non-increasing in $n$, we have 
$$
q_{n,z}(0) \leq z e^{\min_{1\leq k\leq n} S_k}.
$$
Setting $B_{n,j}=\{ -(j+1) <  \min_{1\leq k\leq n} (y+S_k) \leq -j \}$, this implies
\begin{align*}
&\mathbb P_{i,z}(Z_n>0, \tau_y \leq n) \\
& \leq z \mathbb E_{i,0}(e^{\min_{1\leq k\leq n} S_k}; \tau_y \leq n)    \\
&\leq z e^{-y}\sum_{j=1}^{\infty} \mathbb E_{i,0}(e^{\min_{1\leq k\leq n} (y+S_k)}; 
B_{n,j}, 
\tau_y \leq n)\\
&\leq z e^{-y}\sum_{j=1}^{\infty} e^{-j} \mathbb P_{i,0}( \tau_{y+j+1} > n).
\end{align*}
Using the point 2 of Proposition \ref{Prop Init002} we obtain the assertion of the lemma. 
\end{proof}

It is known from the results in \cite{grama_limit_2016-1} that when $y$ is sufficiently large, then $(i,y) \in \supp V$. 
For $(i,y) \in \supp V$, set 
\begin{align} \label{NotUUU001}
U(i,y,z) :=  \mathbb E_{i,y}^{+} q_{\infty,z}(0) = \mathbb P_{i,z,y}^{+} (\cap_{n\geq1} \{ Z_n>0\}).
\end{align}
Theorem \ref{Tsurvival001} is a direct consequence of the following proposition, 
which extends Theorem 1.1 in \cite{GrLauvLePage_2018-SPA} to the case $z>1$.  
\begin{proposition}\label{prop-converg y 001}
Assume Conditions \ref{primitif}-\ref{cond Kersting2017}. Suppose that $i\in \mathbb X$ and $z\in \mathbb N$. 
Then for any $i \in \mathbb X$ the limit as $y\to \infty$ of $V(i,y) U(i,y,z)$ exists and satisfies
\begin{align*}
u(i,z ) :=  \lim_{y\to \infty}  \frac{2}{\sqrt{2\pi} \sigma} V(i,y) U(i,y,z) >0 .
\end{align*}
Moreover,
\begin{align*}
\lim_{y\to \infty} \sqrt{n}   \mathbb P_{i,z} (Z_n>0, X_n=j) = 
u(i,z ) \bs \nu (j).
\end{align*}
\end{proposition}
\begin{proof}
By Lemma \ref{lemma limit 003} and \eqref{NotUUU001}, for $(i,y)\in \supp V$,
\begin{align}\label{UUU0001}
&\lim_{n\to\infty}\sqrt{n} \bb P_{i,z} \left( Z_n > 0 \,,\, X_n = j \,,\, \tau_y > n \right) \nonumber \\
&= \frac{2V(i,y)}{\sqrt{2\pi}\sigma} \bs \nu (j)
\mathbb P_{i,y,z}^{+} (\cap_{n\geq1} \{ Z_n>0\}) \nonumber \\
&= \frac{2V(i,y)}{\sqrt{2\pi}\sigma}  U(i,y,z) \bs \nu (j). 
\end{align}
\begin{align}\label{UUU0002}
 \sqrt{n}   \mathbb P_{i,z} (Z_n>0, X_n=j)
& = \sqrt{n}   \mathbb P_{i,z} (Z_n>0, X_n=j, \tau_y > n)  \nonumber \\
& +\sqrt{n}   \mathbb P_{i,z} (Z_n>0, X_n=j, \tau_y \leq n)  \nonumber \\
&= J_1(n,y) + J_2(n,y).
 \end{align}
By Lemma \ref{lemma bound tau<=n},
\begin{align} \label{UUU0003}
J_2(n,y) \leq c z  e^{-y} (1+ \max\{y,0 \}).
\end{align}
From \eqref{UUU0001}, \eqref{UUU0002} and \eqref{UUU0003}, when $y$ is sufficiently large,
\begin{align} \label{UUU0004}
\limsup_{n\to\infty} \sqrt{n}   \mathbb P_{i,z} (Z_n>0, X_n=j) 
&\leq
\frac{2V(i,y)}{\sqrt{2\pi}\sigma}  U(i,y,z) \bs \nu (j) 
\nonumber \\
&+ c z  e^{-y} (1+ \max\{y,0 \}) <\infty.
\end{align}
Similarly, when $y$ is sufficiently large,
\begin{align} \label{UUU0005}
L_0=\liminf_{n\to\infty} \sqrt{n}   \mathbb P_{i,z} (Z_n>0, X_n=j) 
&\geq
\frac{2V(i,y)}{\sqrt{2\pi}\sigma} U(i,y,z) \bs \nu (j) . 
\end{align}
Since $\bb P_{i,z} \left( Z_n > 0 \,,\, X_n = j \,,\, \tau_y > n \right)$ is non-decreasing in $y$, from 
\eqref{UUU0001} it follows that
the function 
\begin{align} \label{UUU0006}
u(i,y,z ) := \frac{2V(i,y)}{\sqrt{2\pi}\sigma} U(i,y,z)
\end{align}
is non-decreasing in $y$. Moreover, by \eqref{UUU0004} and \eqref{UUU0005}, we deduce that 
$u(i,y,z)$ 
as a function of $y$ is bounded  by $L_0$. Therefore its limit as $y \to \infty$ exists:
$
u(i,z)= \lim_{y\to \infty } u(i,y,z).
$
To prove that $U(i,y,z)=\mathbb E_{i,y}^{+} q_{\infty,z}(0) >0$, 
it is enough to remark that, by \eqref{AAA02d}, it holds $\mathbb E_{i,y}^{+} q^{-1}_{\infty,z}(0) <\infty$.
On the other hand $V(i,y)>0$ for large enough $y$. Therefore $u(i,z)>0$, which proves the first assertion.  
 The second assertion follows immediately from \eqref{UUU0004} and \eqref{UUU0005} by letting $y\to\infty$. 
\end{proof}

\section{Proof of Theorem \ref{T-Loi-limite001}} \label{sec:proofTh1}

All over this section we denote, for $n\geq 1$,
\begin{align*}
T_n=\sup\{ 0\leq k \leq n :   S_k =\inf \{ S_0,\ldots,S_n \}  \},
\end{align*}
and, for $0 \leq k \leq n$,
\begin{align*}
L_{k,n} = \inf_{k\leq j\leq n} \big( S_{j} - S_{k}\big).
\end{align*}
Recall the following identities which will be useful in the proofs:
\begin{align} 
\{T_k=k \}&= \{S_0\geq S_k, S_1\geq S_k,\ldots, S_{k-1}\geq S_{k} \} \label{relat001} \\
\{L_{k,n} > 0 \}&= \{S_{k+1}\geq S_k, S_{k+2} \geq S_k,\ldots, S_n\geq S_k \} \label{relat002} \\
\{T_n=k \} &= \{T_k=k \} \cap \{L_{k,n} >0 \}. \label{relat003}
\end{align}
For any $n\geq 1,$ set
\begin{align} \label{def P_m 001}
P_{n}(i,s,z)=\mathbb E_{i,z} (e^{-e^{-s}\frac{Z_{n}}{e^{S_{n}}} };  Z_{n}>0, L_{0,n} >0 ).
\end{align}
It 
is easy to see that, by the definition \eqref{def-tau001} of $\tau_y$, 
we have $\{\tau_0>n\}=\{L_{0,n}>0\}$, so 
that \eqref{def P_m 001} is equivalent to
\begin{align} \label{def P_m 002}
P_{n}(i,s,z) = \mathbb E_{i,z} (e^{-e^{-s}\frac{Z_{n}}{e^{S_{n}}} }; Z_{n}>0,  \tau_{0} >n   ).
\end{align} 

We prove first a series of auxiliary statements.  
\begin{lemma}\label{lemma bound e power}
Assume Conditions \ref{primitif}-\ref{cond Kersting2017}.
Let $s\geq 0$. 
For any $(i,0)\in \supp V$ and $z\in \mathbb N$, $z\not=0$, there exists a positive random variable $W_{i,z}$ such that
\begin{align*} 
\lim_{n\to \infty} \sqrt{n}P_{n}(i,s,z) = 
2\frac{V(i,0)}{\sqrt{2\pi}\sigma}
P_{\infty}(i,s,z), 
\end{align*}
where
\begin{align*}
P_{\infty}(i,s,z) := \mathbb E^{+}_{i,0,z} \big(e^{-W_{i,z}e^{-s}}; \cap_{p\geq 1} \{ Z_p >0\} \big)\leq 1.
\end{align*}
Moreover, for any $(i,0)\in \supp V$ and $z\in \mathbb N$, $z\not=0$, 
it holds $\mathbb P^{+}_{i,0,z}$-a.s.
\begin{align*} 
\cap_{p\geq 1} \{ Z_p >0\} = \{ W_{i,z} >0\}.   
\end{align*}
For any $(i,0)\not\in \supp V$ and $z\in \mathbb N$, $z\not=0$, 
\begin{align*} 
\lim_{n\to \infty} \sqrt{n}P_{n}(i,s,z) = 0.
\end{align*}
\end{lemma}
\begin{proof}
Denote $Y_n= e^{-e^{-s}\frac{Z_{n}}{e^{S_{n}}} } \mathbf 1 \{ Z_n>0\}$. 
Since $\left(\frac{Z_{n}}{e^{S_{n}}}\right)_{n\geq0}$ 
is a positive $((\mathscr F_n)_{n\geq0}, \mathbb P^{+}_{i,0,z}  )$-martingale,   
its limit, say $W_{i,z}=\lim_{n\to\infty}\frac{Z_{n}}{e^{S_{n}}},$
exists $\mathbb P^{+}_{i,0,z}$-a.s.\ and is non-negative. 
Therefore,  $\mathbb P^{+}_{i,0,z}$-a.s. 
\begin{align*}
\lim_{n\to \infty} Y_n = e^{-e^{-s} W_{i,z}  }  \mathbf 1 \{ \cap_{p\geq 1} \{ Z_p >0\}  \}.
\end{align*}
Now 
the first assertion follows from the Lemma  \ref{lemma limit 003}.

For the second assertion we use a result from Kersting \cite{Kersting_2017}
stated in a more general setting of branching processes with varying environment. 
To apply it we shall condition with respect to the environment $(X_n)_{n\geq 0}$, so that one can consider that the environment is fixed. 
The condition  (A) in Kersting \cite{Kersting_2017} is obviously verified because of the Condition \ref{cond Kersting2017}.
Moreover, according to Lemma \ref{lemma boundprob001}, the extinction probability satisfies $P_{i,0,z}^{+}$-a.s. 
\begin{align*}
\mathbb P_{i,0,z}^{+} (  \cup_{p\geq 1} \{ Z_p =0\} \big| (X_n)_{n\geq 1} ) <1.
\end{align*}
By Theorem 2 in \cite{Kersting_2017}, this implies that, $P_{i,0,z}^{+}$-a.s. 
\begin{align*} 
\mathbb P_{i,0,z}^{+} (  \cup_{p\geq 1} \{ Z_p =0\} \big| (X_n)_{n\geq 1} )
=\mathbb P_{i,0,z}^{+} (  W_{i,z} =0 \big| (X_n)_{n\geq 1} ).
\end{align*}
Since $P_{i,0,z}^{+}$-a.s.\ $ \cup_{p\geq 1} \{ Z_p =0\}  \subset \{ W_{i,z} = 0\},$
we obtain the second assertion. 

The third one follows from the point 1 of Proposition \ref{Prop Init002}, since $V(i,0)=0$.
This ends the  proof of the lemma. 
 \end{proof}

Remark that $P_{\infty}(i,s,z)$ can be rewritten as
\begin{align*} 
P_{\infty}(i,s,z) = \mathbb E^{+}_{i,0,z} \big(e^{-W_{i,z}e^{-s}}; W_{i,z} >0 \big).
\end{align*}
This shows that $P_{\infty}(i,s,z)$ is the Laplace transformation at $e^{-s}$
of a measure on $\mathbb R_{+}$ 
which assigns the mass $0$ to the set $\{0\}$.

We will need the following lemma.
\begin{lemma}\label{lemma bound n3/2}
There exists a constant $c$ such that
\begin{align*} 
\sup_{i\in \mathbb X}\mathbb P_{i,z} \big(   T_n=n,  Z_n >0 \big) \leq c\frac{ z}{n^{3/2}}.
\end{align*}
\end{lemma}
\begin{proof}
Since $T_n$ is a function only on the environment $\left( X_k \right)_{k\geq 0}$, 
conditioning with respect to $\left( X_k \right)_{k\geq 0}$, we have
\begin{align*} 
\mathbb P_{i,z} \big(   T_n=n,  Z_n >0 \big)
= \mathbb E_{i,z} \big( q_{n,z}(0); T_n=n \big),
\end{align*}
with $q_{n,z}(0)$ defined by \eqref{SURV000-in0}. 
Using the bound \eqref{Ineqz001} we obtain
\begin{align} \label{EqT44-001}
\mathbb P_{i,z} \big(   T_n=n,  Z_n >0 \big)
\leq z \mathbb E_{i} \big( e^{S_n}; T_n=n \big).
\end{align}
By \eqref{relat001} and the duality (Lemma \ref{dualityBP}),
\begin{align} \label{EqT44-002}
\mathbb E_{i} \big( e^{S_n}; T_n=n \big) 
&=  \mathbb E_{\bs \nu}^* \big( e^{-S_n^*}; \tau^*_0>n \big) \frac{1}{\bs \nu(i)} \nonumber \\
&\leq c \mathbb E_{\bs \nu}^* \big( e^{-S_n^*}; \tau^*_0>n \big).
\end{align}
Using the local limit theorem for the dual Markov chain (see Proposition \ref{LLTC}) and following the proof of Lemma \ref{lemma sum eS_n } we obtain
\begin{align} \label{EqT44-003}
\mathbb E_{\bs \nu}^* \big( e^{-S_n^*}; \tau^*_0>n \big) \leq \frac{c}{n^{3/2}}. 
\end{align}
From \eqref{EqT44-001}, \eqref{EqT44-002} and \eqref{EqT44-003} the assertion follows. 
\end{proof}

The key point of the proof of Theorem \ref{T-Loi-limite001} is the following statement.
\begin{proposition} \label{PropTT12-003}
Assume Conditions \ref{primitif}-\ref{cond Kersting2017}.
For any $i\in \mathbb X$, $s\in \mathbb R$ and $z\in \mathbb N$, $z\not=0$,
\begin{align*} 
&\lim_{n\to \infty} \sqrt{n}\mathbb E_{i,z} (  e^{-e^{-s}\frac{Z_n}{e^{S_n}} }; Z_n>0) \\
&=\frac{2}{\sqrt{2\pi}\sigma}\sum_{k=1}^{\infty} 
\mathbb E_{i,z} \big(V(X_k,0) \mathbf 1_{\supp V} (X_k,0) P_{\infty}(X_k,s+S_k,Z_k);\\
&\qquad\qquad\qquad\qquad\qquad\qquad\qquad\qquad\qquad Z_k>0, T_k=k
\big)\\
& =: u(i,z, e^{-s}).
\end{align*}
 With $s=+\infty$,
\begin{align*}
&\lim_{n\to \infty} \sqrt{n}\mathbb P_{i,z} ( Z_n>0)  = u(i,z),\\
\end{align*}
where 
\begin{align*}
&u(i,z)=\frac{2}{\sqrt{2\pi}\sigma}\sum_{k=1}^{\infty} \mathbb E_{i,z} \big(
V(X_k,0) \mathbf 1_{\supp V} (X_k,0)
\mathbb P^{+}_{X_k,0,Z_k} ( W_{i,z} >0 );\\
&\qquad\qquad\qquad\qquad\qquad\qquad\qquad\qquad\qquad Z_k>0, T_k=k\big) >0
\end{align*}
is defined in Theorem \ref{Tsurvival001}.
\end{proposition}
\begin{proof}
 Using \eqref{relat003}, one has
\begin{align} \label{eqPropTT12-001}
 \mathbb E_{i,z} (e^{-e^{-s}\frac{Z_n}{e^{S_n}} }; Z_n>0) 
& =\sum_{k=0}^{n-1} \mathbb E_{i,z} \big( e^{-e^{-s}\frac{Z_n}{e^{S_n}} }; Z_n>0,   T_k=k, L_{k,n} >0 \big) \nonumber\\
&+ \mathbb E_{i,z} \big( e^{-e^{-s}\frac{Z_n}{e^{S_n}} }; Z_n>0,   T_n=n \big)\nonumber\\
&=J_1(n) + J_2(n).
\end{align}
By Lemma \ref{lemma bound n3/2},
\begin{align}\label{eqPropTT12-002}
\limsup_{n\to\infty}\sqrt{n} J_2(n) \leq \lim_{n\to\infty} \sqrt{n} \mathbb P_{i,z} \big(   T_n=n,  Z_n >0 \big) = 0.
\end{align}
We now deal with the term $J_1(n)$. 
We shall make use of the notation $P_{n}(i,y,z)$ defined in \eqref{def P_m 002}.
By the Markov property (conditioning 
with respect to $\mathscr F_k = \sigma \{ X_0,Z_0, \ldots, X_k, Z_k \}$), 
and using \eqref{relat002}, we obtain
\begin{align*}
&\mathbb E_{i,z} \big( e^{-e^{-s}\frac{Z_n}{e^{S_n}} }; Z_n>0,   T_k=k, L_{k,n} >0 \big) \\
&\qquad =
\mathbb E_{i,z} \big( P_{n-k}(X_{k},s+S_k,Z_{k}) ;  T_k=k,  Z_k >0 \big). 
\end{align*}
Therefore 
\begin{align*}
J_1(n) 
=\sum_{k=0}^{n-1} \frac{1}{\sqrt{n-k}} 
\mathbb E_{i,z} \big( \sqrt{n-k}  P_{n-k}(X_{k},s+S_k,Z_{k}) ;  T_k=k,  Z_k >0 \big). 
\end{align*}
Denote for brevity 
\begin{align*}
E_k=
\mathbb E_{i,z} 
\big( 
\sqrt{n-k}  P_{n-k}(X_{k},s+S_k,Z_{k}) ;  T_k=k,  Z_k >0 \big).
\end{align*}
It is easy to see that, with some $   l\leq n$,
\begin{align}\label{eqPropTT12-003}
\sqrt{n}J_1(n)
&= \sum_{k=0}^{l}   \frac{\sqrt{n}}{\sqrt{n-k}} E_k 
 + \sum_{k=l+1}^{n-1}  \frac{\sqrt{n}}{\sqrt{n-k}} E_k \nonumber\\ 
&= J_{11}(n,l) + J_{12}(n,l).
\end{align}
For $J_{12}(n,l)$, we have, using \eqref{def P_m 002},
\begin{align*}
J_{12}(n,l) &= \sum_{k=l+1}^{n-1}  \frac{\sqrt{n}}{\sqrt{n-k}} E_k \\
&\leq \sum_{k=l+1}^{n-1}  \frac{\sqrt{n}}{\sqrt{n-k}} \mathbb E_{i,z} 
\big( \sqrt{n-k}  \mathbb P_{X_{k}} (\tau_0>n-k) ;  T_k=k,  Z_k >0 \big). 
\end{align*}
Using point 2 of Proposition \ref{Prop Init002} and Lemma \ref{lemma bound n3/2}, 
\begin{align*}
J_{12}(n,l)
&\leq c\sum_{k=l+1}^{n-1}  \frac{\sqrt{n}}{\sqrt{n-k}} \mathbb P_{i,z}  \big(  T_k=k,  Z_k >0 \big) \\
&\leq c z\sum_{k=l+1}^{n-1}  \frac{\sqrt{n}}{\sqrt{n-k}} \frac{ 1}{k^{3/2}}\\
&\leq c z\bigg(  \frac{1}{\sqrt{n}} + \frac{2}{\sqrt{l}} \bigg),
\end{align*}
where to bound the second line we split the summation in two parts for $k > k/2$ and $k\leq k/2$. 
Let $\varepsilon >0$ be arbitrary. 
Then there exists $n_{\varepsilon, z}$ such that, for $n\geq l\geq n_{\varepsilon, z}$, 
\begin{align}\label{eqPropTT12-005}
J_{12}(n,l) \leq \varepsilon.
\end{align}
For $J_{11}(n,l)$, we have
\begin{align*}
J_{11}(n,l) = \sum_{k=0}^{l}   \frac{\sqrt{n}}{\sqrt{n-k}} \mathbb E_{i,z} 
\big( 
\sqrt{n-k}  P_{n-k}(X_{k},s+S_k,Z_{k}) ;  T_k=k,  Z_k >0 \big). 
\end{align*}
Since $l$ is fixed, taking the limit as $n\to \infty$, by Lemmata \ref{lemma bound e power} and \ref{lemma bound n3/2}, 
\begin{align} \label{eqPropTT12-008}
\lim_{n\to\infty} J_{11}(n,l) 
&= \sum_{k=0}^{l} \mathbb E_{i,z} \big( P_{\infty}(X_k,s+S_k,Z_k) ;  T_k=k,  Z_k >0 \big) \nonumber\\
&\leq \sum_{k=0}^{l} \mathbb P_{i,z} \big( T_k=k,  Z_k >0 \big)  \nonumber\\
&\leq \sum_{k=0}^{\infty} \frac{cz}{k^{3/2}}  \leq c z .
\end{align}
Since $\varepsilon$ is arbitrary, 
from \eqref{eqPropTT12-003}, \eqref{eqPropTT12-005} and \eqref{eqPropTT12-008},
taking the limit as $l\to\infty$, 
we deduce that 
\begin{align*}
\lim_{n\to\infty} \sqrt{n}J_{1}(n) 
&= \sum_{k=0}^{\infty} \mathbb E_{i,z} \big( P_{\infty}(X_k,s+S_k,Z_k) ;  T_k=k,  Z_k >0 \big).
\end{align*}
From this and \eqref{eqPropTT12-001}, \eqref{eqPropTT12-002} we deduce the first assertion of the proposition. 
The second one is proved in the same way.
\end{proof}

Now we proceed to prove Theorem \ref{T-Loi-limite001}. Denote by
$$
\mu_{n,i,z,j} (B)  = \mathbb P_{i,z}
   \left( Z_ne^{-S_n} \in B, X_n=j   \big| Z_n>0 \right)
$$
the joint law law of $\frac{Z_n}{e^{S_n}}$ and $X_n=j$ 
given $Z_n>0$ under $\mathbb P_i$, where $B$ is any Borel set of $\mathbb R_+.$
Set for short
$$
\mu_{n,i,z} (B)  = \mathbb P_{i,z}
   \left( Z_ne^{-S_n} \in B   \big| Z_n>0 \right).
$$
We shall prove that the sequence $(\mu_{n,i,z})_{n\geq1}$ is convergent in law. 
For this we use the convergence of the corresponding Laplace transformations: 
\begin{align*}
\mathbb E_{i,z} (  e^{-t\frac{Z_n}{e^{S_n}} } \big| Z_n>0) &= 
\frac{\sqrt{n}\mathbb E_{i,z} (  e^{-t\frac{Z_n}{e^{S_n}} }; Z_n>0)}
{\sqrt{n}\mathbb P_{i,z} ( Z_n>0)}. 
\end{align*}
By Proposition \ref{PropTT12-003}, we see that, with $t=e^{-s}$,
\begin{align*}
\lim_{n\to \infty} \mathbb E_{i,z} (  e^{-e^{-s}\frac{Z_n}{e^{S_n}} } \big| Z_n>0) 
= 
 \frac{u(i,z, e^{-s}) } {u(i,z) }. 
\end{align*}
It is obvious that $u(i,z, e^{-s})$ is also a Laplace transformation
at $t=e^{-s}$
of a measure on $\mathbb R_{+}$ 
which assigns the mass $0$ to the set $\{0\}$
and that $u(i,z)$ is the total mass of this measure.  
Therefore the ratio
$ \frac{u(i,z, e^{-s}) } {u(i,z) }$ is the Laplace transformation of a probability measure $\mu_{i,z}$ on $R_{+}$ 
such that $\mu_{i,z}(\{0\})=0.$ 
\section{Proof of Theorems \ref{T-Loi-limite003} and \ref{T-Yaglom}} \label{sec: 44444}
Recall that by the properties of the harmonic function, for any $i\in \mathbb X$ there exists $y_i\geq 0$
such that $(i,y)\in \supp V$ for any $y\geq y_i.$

First we prove the following auxiliary statement.
\begin{lemma}\label{lemmaTL003-001}
Assume Conditions \ref{primitif}-\ref{cond Kersting2017}.
Let $i\in \mathbb X$ and $z\in \mathbb N$, $z\not = 0$.
For any $\theta \in (0,1)$ and
$y\geq 0$ large enough such that $(i,y)\in \supp V$,
\begin{align*}
\lim_{m\to\infty}\lim_{n\to\infty} \sqrt{n} \mathbb P_{i,z}  \left( Z_m>0, Z_{[\theta n]}=0, \tau_y>n  \right) = 0.
\end{align*}
\end{lemma}
\begin{proof}
Let $m, n\geq 1$ be such that $[\theta n] >m$. Then 
\begin{align*}
J_{m,n}(\theta,y)&:=\mathbb P_{i,z}  \left( Z_m>0, Z_{[\theta n]}=0, \tau_y>n  \right) \\
&= \mathbb P_{i,z}  \left( Z_m>0, \tau_y>n  \right) -\mathbb P_{i,z}  \left(Z_{[\theta n]}>0, \tau_y>n  \right)\\
&= \mathbb E_{i,z} 
\big( \mathbb P_{i,z}  \left( Z_m>0 | X_1,\ldots, X_m \right) \\
&\qquad\qquad - \mathbb P_{i,z}  \left(Z_{[\theta n]}>0 | X_1,\ldots, X_{[\theta n]} \right); \tau_y>n   \big)\\
&\leq \mathbb E_{i,z} \big( \left| q_{m,z}(0) - q_{[\theta n],z}(0) \right|; \tau_y>n   \big).
\end{align*}
Denote $P_n(i,y)=\mathbb P_{i}(\tau_y>n).$
 By the Markov property
\begin{align*}
&\mathbb E_{i,z} \big( \left| q_{m,z}(0) - q_{[\theta n],z}(0) \right|; \tau_y>n   \big)\\
&=\mathbb E_{i} \big( \left| q_{m,z}(0) - q_{[\theta n],z}(0) \right|  
P_{n-[\theta n]}(X_{[\theta n]}, y+ S_{[\theta n]}); 
\tau_y>[\theta n]   \big).
\end{align*}
Using point 2 of Proposition \ref{Prop Init002} and the
point 3 of Proposition \ref{PropF-hamon001}, 
on the set $\{ \tau_y>[\theta n] \}$,
\begin{align*}
 P_{n-[\theta n]}(X_{[\theta n]}, y+ S_{[\theta n]},z) 
\leq c\frac{1+y+ S_{[\theta n]}}{\sqrt{n-[\theta n]}} 
 \leq c\frac{1+ V(X_{[\theta n]}, y+ S_{[\theta n]})  }{\sqrt{(1-\theta) n}  }.  
\end{align*}
Therefore
\begin{align*}
&\sqrt{n} J_{m,n}(\theta,y) \\
&\leq 
\mathbb E_{i} \big( \left| q_{m,z}(0) - q_{[\theta n],z}(0) \right|  
\sqrt{n}P_{n-[\theta n]}(X_{[\theta n]}, y+ S_{[\theta n]}); 
\tau_y>[\theta n]   \big)\\
&\leq \frac{c}{\sqrt{1-\theta}  }
 \mathbb E_{i} \big( \left| q_{m,z}(0) - q_{[\theta n],z}(0) \right|  ; \tau_y>[\theta n]  \big)   \\
&\quad + \frac{c}{\sqrt{1-\theta}  }
 \mathbb E_{i} \big( \left| q_{m,z}(0) - q_{[\theta n],z}(0) \right|  
V(X_{[\theta n]}, y+ S_{[\theta n]})  ; \tau_y>[\theta n]  \big).
\end{align*}
Using the bound \eqref{bound for q_{n,z}(s)}  and again the point 2 of Proposition \ref{Prop Init002}, 
\begin{align*}
\limsup_{n\to\infty}\mathbb E_{i} \big( \left| q_{m,z}(0) - q_{[\theta n],z}(0) \right|  ; \tau_y>[\theta n]  \big)
\leq \lim_{n\to\infty} \mathbb P_{i} \big( \tau_y>[\theta n]  \big) = 0.
\end{align*}
If $(i,y)\not\in \supp V$,
\begin{align*}
& \mathbb E_{i} \big( \left| q_{m,z}(0) - q_{[\theta n],z}(0) \right| V(X_{[\theta n]}, y+ S_{[\theta n]})  ; \tau_y>[\theta n]  \big) \\
&\leq 
\mathbb E_{i} \big( V(X_{[\theta n]}, y+ S_{[\theta n]})  ; \tau_y>[\theta n]  \big) \\
&= V(i,y)=0.
\end{align*}
If $(i,y)\in \supp V$, changing the measure by \eqref{changemes002}, 
we have
\begin{align*}
& \mathbb E_{i} \big( \left| q_{m,z}(0) - q_{[\theta n],z}(0) \right| V(X_{[\theta n]}, y+ S_{[\theta n]})  ; \tau_y>[\theta n]  \big) \\
&= V(i,y) \mathbb E_{i,y}^+ \left| q_{m,z}(0) - q_{[\theta n],z}(0) \right|.
\end{align*}
Taking the limit as $n\to\infty$ and then as $m\to\infty$, by Lemma \ref{lemma q_n conv vers q },
\begin{align*}
&\lim_{m\to\infty}\lim_{n\to\infty} \mathbb E_{i} \big( \left| q_{m,z}(0) - q_{[\theta n],z}(0) \right| V(X_{[\theta n]}, y+ S_{[\theta n]})  ; \tau_y>[\theta n]  \big) =0.
\end{align*}
Taking into account the previous bounds we obtain the assertion of the lemma. 
\end{proof}

\begin{proof}[Proof of Theorem \ref{T-Loi-limite003}]
Let $i,j\in \mathbb X$, $y\geq 0$, $z\in \mathbb N$, $z\not=0$, $t\in \mathbb R$.  
Then, for any $n\geq 1$ and $y \geq 0,$
\begin{align} \label{proofT1.3-001}
&\mathbb P_{i,z}  \left( \frac{S_n}{\sqrt{n}\sigma} \leq t, X_n=j, Z_n>0  \right) \nonumber \\
&= \mathbb P_{i,z}  \left( \frac{S_n}{\sqrt{n}\sigma} \leq t, X_n=j, Z_n>0, \tau_y>n  \right)\nonumber \\
&+ \mathbb P_{i,z}  \left( \frac{S_n}{\sqrt{n}\sigma} \leq t, X_n=j, Z_n>0, \tau_y\leq n  \right)\nonumber \\
&=I_{1}(n,y) + I_{2}(n,y).
\end{align}
By Lemma \ref{lemma bound tau<=n}, 
\begin{align} \label{eq-proofT1.3-301}
 \sqrt{n} I_{2}(n,y) \leq 
 \sqrt{n} P_{i,z}(Z_n>0, \tau_y \leq n)  
 \leq c  z e^{-y}(1+y).
\end{align}

In the sequel we study $I_{1}(n,y)$.
Let $\theta \in (0,1)$ be arbitrary. 
We  decompose $I_{1}(n,y)$ into two parts: 
\begin{align} \label{proofT1.3-002}
I_{1}(n,y) 
&=\mathbb P_{i,z}  \left( \frac{S_n}{\sqrt{n}\sigma} \leq t, X_n=j, Z_{[\theta n]}>0, Z_n>0, \tau_y>n  \right) \nonumber \\
&=\mathbb P_{i,z}  \left( \frac{S_n}{\sqrt{n}\sigma} \leq t, X_n=j, Z_{[\theta n]}>0, \tau_y>n  \right) \nonumber \\
&- \mathbb P_{i,z}  \left(  \frac{S_n}{\sqrt{n}\sigma} \leq t, X_n=j, Z_{[\theta n]}>0, Z_n=0, \tau_y>n  \right) \nonumber \\
&=I_{11}(n,\theta,y)  - I_{12}(n,\theta,y). 
\end{align}

In the following lemma we prove that the second term $\sqrt{n}I_{12}(n,\theta,y)$ vanishes as $n\to \infty.$ 
\begin{lemma}\label{lemma-I2-002}
Assume Conditions \ref{primitif}-\ref{cond Kersting2017}.
For any $i,j\in \mathbb X$, $z\in \mathbb N$, $z\not= 0$, and 
$y\geq 0$ sufficiently large,
\begin{align} \label{eqlemma-I2-002}
\lim_{n\to \infty} \sqrt{n} 
\mathbb P_{i,z}  \left( \frac{S_n}{\sqrt{n}\sigma} \leq t, X_n=j, Z_{[\theta n]}>0, Z_n=0, \tau_y>n  \right) 
= 0. 
\end{align}
\end{lemma}
\begin{proof}
Obviously,
\begin{align}\label{eqPPRR-001}
| I_{12}(n,\theta,y)| 
&=\left| \mathbb P_{i,z}  \left( \frac{S_n}{\sqrt{n}\sigma} \leq t, X_n=j, Z_{[\theta n]}>0, Z_n=0, \tau_y>n  \right) \right| \nonumber \\
&\leq \mathbb P_{i,z}  \left( X_n=j, Z_{[\theta n]}>0, Z_n=0, \tau_y>n  \right) \nonumber \\
&=\mathbb P_{i,z}  \left( X_n=j, Z_{[\theta n]}>0, \tau_y>n  \right) \nonumber  \\
&\quad - \mathbb P_{i,z}  \left( X_n=j, Z_n>0, \tau_y>n  \right).
\end{align}
As in \eqref{UUU0001}, 
choosing $y\geq 0$ such that $(i,y) \in \supp V$,
we have 
\begin{align} \label{eqPPRR-002}
\lim _{n\to\infty} \sqrt{n} \mathbb P_{i,z}  \left( X_n=j, Z_n>0, \tau_y>n  \right) = \frac{2V(i,y)}{\sqrt{2\pi}\sigma}  U(i,y,z) \bs \nu (j).
\end{align}
We shall prove that for any $\theta \in (0,1)$,
\begin{align} \label{eqPPRR-005}
\lim _{n\to\infty} \sqrt{n} \mathbb P_{i,z}  \left( X_n=j, Z_{[\theta n]}>0, \tau_y>n  \right) = \frac{2V(i,y)}{\sqrt{2\pi}\sigma}  U(i,y,z) \bs \nu (j).
\end{align}
For any $m\geq 1$ and $n$ such that $[\theta n] >m$, 
\begin{align} \label{eqPPRR-006}
&\mathbb P_{i,z}  \left( X_n=j, Z_{[\theta n]}>0, \tau_y>n  \right) \nonumber  \\
&=\mathbb P_{i,z}  \left( X_n=j, Z_m>0, Z_{[\theta n]}>0, \tau_y>n  \right)  \nonumber \\
&=\mathbb P_{i,z}  \left( X_n=j, Z_m>0, \tau_y>n  \right)  \nonumber \\
&\quad - \mathbb P_{i,z}  \left( X_n=j, Z_m>0, Z_{[\theta n]}=0, \tau_y>n  \right).
\end{align}
By Lemma \ref{lemmaE+}, 
\begin{align*} 
&\lim_{n\to\infty} \sqrt{n} \mathbb P_{i,z}  \left( X_n=j, Z_m>0, \tau_y>n  \right) \nonumber \\
&= \frac{2V(i,y)}{\sqrt{2\pi}\sigma}  \bb P_{i,z,y}^+ \left( Z_m>0 \right) \bs \nu (j).
\end{align*}
Taking the limit as $m\to\infty$, by \eqref{NotUUU001}, we have 
\begin{align} \label{eqPPRR-008}
&\lim_{m\to\infty} \lim_{n\to\infty} \sqrt{n} \mathbb P_{i,z}  \left( X_n=j, Z_m>0, \tau_y>n  \right) \nonumber \\
&= \frac{2V(i,y)}{\sqrt{2\pi}\sigma}  \mathbb P_{i,z,y}^{+} (\cap_{m\geq1} \{ Z_m>0\}) \bs \nu (j)  \nonumber  \\
&= \frac{2V(i,y)}{\sqrt{2\pi}\sigma}  U(i,y,z) \bs \nu (j).
\end{align}
By Lemma \ref{lemmaTL003-001},
\begin{align*}
& \limsup_{m\to\infty} \limsup_{n\to\infty} \sqrt{n} \mathbb P_{i,z}  \left( X_n=j, Z_m>0, Z_{[\theta n]}=0, \tau_y>n  \right) \\
&\leq \limsup_{m\to\infty} \lim_{n\to\infty} \sqrt{n} \mathbb P_{i,z}  \left( Z_m>0, Z_{[\theta n]}=0, \tau_y>n  \right) = 0,
\end{align*}
which together with \eqref{eqPPRR-006} and \eqref{eqPPRR-008}  proves \eqref{eqPPRR-005}.
From \eqref{eqPPRR-001}, \eqref{eqPPRR-002} and \eqref{eqPPRR-005} we obtain the assertion of the lemma.
\end{proof}
To handle the term $I_{11}(n,\theta,y)$ we choose any $m$ satisfying $1\leq m \leq [\theta n]$ and split it into two parts:
\begin{align} \label{proofT1.3-003}
I_{11}(n,\theta,y)     
&= \mathbb P_{i,z}  \left( \frac{S_n}{\sqrt{n}\sigma} \leq t, X_n=j, Z_m>0, \tau_y>n  \right) \nonumber \\
&-  \mathbb P_{i,z}  \left( \frac{S_n}{\sqrt{n}\sigma} \leq t, X_n=j, Z_m>0,  Z_{[\theta n]}=0, \tau_y>n  \right) \nonumber \\
& = I_{111}(n,m,y) - I_{112}(n,m,\theta,y).
\end{align}
By Lemma \ref{lemmaTL003-001},  
we have,
\begin{align} \label{lemmaTL003-111}
&\limsup_{m\to\infty}  \limsup_{n\to\infty}
\sqrt{n} I_{112}(n,m,\theta,y) \nonumber \\
&\leq \lim_{m\to\infty} \lim_{n\to\infty} \sqrt{n}\mathbb P_{i,z}  \left(Z_m>0,  Z_{[\theta n]}=0, \tau_y>n  \right) =0.
\end{align}

The following lemma gives a limit for $\sqrt{n}I_{111}(n,m,y)$ as $n\to\infty$ and $m\to\infty$.
\begin{lemma}\label{lemma-I2-104}
Assume Conditions \ref{primitif}-\ref{cond Kersting2017}. 
Suppose that $i,j\in \mathbb X$, $z\in \mathbb N$, $z\not = 0$ and $t\in \mathbb R.$
Then, for any $y\geq 0$ sufficiently large,  
\begin{align} \label{eqlemma-I2-104}
\lim_{m\to \infty} & \lim_{n\to \infty}  \sqrt{n} 
\mathbb P_{i,z}  \left( \frac{S_n}{\sqrt{n}\sigma} \leq t, X_n=j, Z_m>0, \tau_y>n  \right) \nonumber \\ 
&= \frac{2\Phi^{+} (t)}{\sqrt{2\pi} \sigma} V(i,y) U(i,y,z) \bs \nu (j).
\end{align}
\end{lemma}
\begin{proof}
Without loss of generality we can assume that $n\geq2m.$
Let $y \geq 0$ be so large that $(i,y)\in \supp V.$
Set $t_{n,y}=t+\frac{y}{\sqrt{n}\sigma}$.
Then $I_{111}(n,m,y)$ can be rewritten as 
\begin{align*}
I_{111}(n,m,y)= \mathbb P_{i,z}  \left( \frac{y+S_n}{\sqrt{n}\sigma} \leq t_{n,y}, X_n=j, Z_m>0, \tau_y>n  \right).
\end{align*}
If $t<0$, we have $t_{n,y}<0$, for $n$ large enough, then the assertion of the lemma becomes obvious since 
$I_{111}(n,m,y)= 0 =\Phi^{+} (t)$. 
Therefore, it is enough to assume that $t \geq 0$. 
To find the asymptotic of $I_{111}(n,m,y)$ we introduce the following notation: for $y',t'\in \mathbb R$ and $1\leq k=n-m\leq n$,
\begin{align*}
\Phi_{k}(i,y',t') = \mathbb P_{i}\Big(\frac{y'+S_k}{\sigma\sqrt{k}}  \leq t'; \tau_{y'}>k     \Big).
\end{align*}
Set $t_{n,m,y}=t_{n,y}\frac{\sqrt{n}}{\sqrt{n-m}}= (t+\frac{y}{\sqrt{n}\sigma})\frac{\sqrt{n}}{\sqrt{n-m}}.$
By the Markov property,  
\begin{align*}
I_{111}(n,m,y)= \mathbb E_{i,z} \big( \Phi_{n-m}(X_m,y+S_m, t_{n,m,y})      ; Z_m>0, \tau_y>m\big).
\end{align*}
Since $t_{n,m,y} \leq \sqrt{2} (t +y/\sigma)$,
by Proposition \ref{Prop-CondLimTh}, there exists an $\varepsilon>0$ such that 
for any $(i,y')\in \mathbb X\times \mathbb R$,
\begin{align}\label{eqAAA000}
& \sqrt{n-m}\abs{\Phi_{n-m}(i,y',t_{n,m,y}) -\frac{2V(i,y')}{\sqrt{2\pi (n-m)} \sigma}\Phi^{+} \Big(t_{n,m,y} \Big)   } \nonumber \\
&\leq c_{\varepsilon,t,y}\frac{1+\max\{ y',0  \}^2  }{(n-m)^{\varepsilon}}.
\end{align}
Therefore, using \eqref{eqAAA000},
\begin{align*}
&A_n(m,y):=\\
&\abs{I_{111}(n,m,y) - \mathbb E_{i,z} \big(   \frac{2V(X_m,y+S_m)}{\sqrt{2\pi (n-m)} \sigma}\Phi^{+} \big(t_{n,m,y}\big)    ; Z_m>0, \tau_y>m\big)}
 \\ 
 &\leq c_{\varepsilon,t,y}\frac{1+\mathbb E_{i,z} \max\{ y+S_m,0  \}^2  }{(n-m)^{1/2+\varepsilon}}.
\end{align*}
This implies that
\begin{align}\label{eqAA001}
\lim_{n\to \infty} \sqrt{n} A_n(m,y) = 0.
\end{align}
Note that with $m$ and $y\geq 0$ fixed, we have $t_{n,m,y} \to t$ as $n\to\infty$. Therefore
\begin{align} \label{eqAA002}
\lim_{n\to +\infty}
&\sqrt{n} \mathbb E_{i,z} \big(   \frac{2V(X_m,y+S_m)}{\sqrt{2\pi k} \sigma}\Phi^{+} \big(t_{n,m}\big)    ; Z_m>0, \tau_y>m\big) \nonumber \\
&=  \frac{2\Phi^{+} (t)}{\sqrt{2\pi} \sigma} \mathbb E_{i,z} \big(   V(X_m,y+S_m) ; Z_m>0, \tau_y>m\big).
\end{align}
When $(i,y) \in \supp V$, the change of measure \eqref{changemes001}, gives
\begin{align} \label{eqAA005}
\mathbb E_{i,z} & \big( V(X_m,y+S_m) ; Z_m>0, \tau_y>m\big) \nonumber \\
&\qquad\qquad\qquad\qquad\qquad = V(i,y)\mathbb P_{i,y,z}^{+} (Z_m>0).
\end{align}
Since $\mathbb P_{i,y,z}^{+} (Z_m>0) = \mathbb E_{i,y}^{+} q_{m,z}(0)$, using Lemma \ref{lemma q_n conv vers q },
\begin{align*}
\lim_{m\to +\infty} \mathbb P_{i,y,z}^{+} (Z_m>0) = \lim_{m\to +\infty} \mathbb E_{i,y}^{+} q_{m,z}(0)= \mathbb E_{i,y}^{+} q_{\infty,z}(0) = U(i,y,z),
\end{align*}
which, together with \eqref{eqAA001}, \eqref{eqAA002} and \eqref{eqAA005}, gives  
\begin{align*}
\lim_{m\to \infty} \lim_{n\to \infty} \sqrt{n} I_{111}(n,m,y) = \frac{2\Phi^{+} (t)}{\sqrt{2\pi} \sigma} V(i,y) U(i,y,z) \bs \nu (j),
\end{align*}
which proves the assertion of the lemma for $t\geq 0$. 
\end{proof}

We now perform the final assembling.  
From \eqref{proofT1.3-001}, \eqref{eq-proofT1.3-301}, \eqref{proofT1.3-002} and \eqref{eqlemma-I2-002}, 
we have, for any $i,j\in \mathbb X$, $z\in \mathbb N$, $z\not= 0$, $t\in \mathbb R$ and $y$ sufficiently large,   
\begin{align} \label{proofT1-fin001}
&\lim_{n\to \infty}\bigg| \sqrt{n}\mathbb P_{i,z}  \left( \frac{S_n}{\sqrt{n}\sigma} \leq t, X_n=j, Z_n>0  \right)
- \sqrt{n} I_{11}(n,\theta,y) \bigg| \nonumber\\
&
\leq  c  z e^{-y}(1+y).
\end{align}
From 
\eqref{proofT1.3-003}, \eqref{lemmaTL003-111}, \eqref{eqlemma-I2-104}
we obtain
\begin{align} \label{proofT1-fin002}
\lim_{m\to \infty}\lim_{n\to \infty} \sqrt{n} I_{11}(n,\theta,y) 
= \frac{2\Phi^{+} (t)}{\sqrt{2\pi} \sigma} V(i,y) U(i,y,z) \bs \nu (j).
\end{align}
From \eqref{proofT1-fin001} and \eqref{proofT1-fin002}, 
taking consecutively the limits as $n\to \infty$ and $m\to \infty$, 
we obtain, for any $i,j \in \bb X$, $z \in \bb N$, $z\not = 0$, 
$t \in \mathbb R$ and $y\geq 0$ sufficiently large, such that $(i,y)\in\supp V$,
\begin{align*}
\lim_{n\to\infty} \bigg|
&\sqrt{n} \bb P_{i,z}
\left(  \frac{ S_n}{\sigma\sqrt{n}} \leq t, X_n=j,  Z_n>0 \right) 
 - \frac{2\Phi^{+} (t)}{\sqrt{2\pi} \sigma} V(i,y) U(i,y,z) \bs \nu (j) \bigg| \\
&
\leq  c  z e^{-y}(1+y).
\end{align*}
Taking the limit as $y\to \infty$,
we obtain, for any $i,j \in \bb X$, $z \in \bb N$, $z\not = 0$ and $t \in \mathbb R$,
$$
\lim_{n\to\infty} \sqrt{n} \bb P_{i,z} 
\left(  \frac{ S_n}{\sigma\sqrt{n}} \leq t, X_n=j,  Z_n>0 \right) 
= \Phi^{+} (t)  \bs \nu (j) u(i,z),
$$
where $u(i,z)>0$ is defined in Proposition \ref{prop-converg y 001}.
This proves the first assertion of the theorem. 
The second assertion is obtained from the first one by using Theorem \ref{Tsurvival001}.
\end{proof}

\begin{proof}[Proof of Theorem \ref{T-Yaglom}]
The assertion of Theorem \ref{T-Yaglom} is easily obtained from Theorem \ref{T-Loi-limite003}.
Let $\varepsilon>0$ be arbitrary. By simple computations,
\begin{align*}
\bb P_{i,z} &\left( \abs{ \frac{\log Z_n}{\sigma\sqrt{n}} - \frac{S_n}{\sqrt{n}\sigma}   } \geq \varepsilon , X_n = j, Z_n>0 \right) \\
&=\bb P_{i,z} \left( \frac{Z_n}{e^{S_n}}   \geq e^{\varepsilon\sqrt{n}\sigma} , X_n = j, Z_n>0 \right) \\
&+\bb P_{i,z} \left( \frac{Z_n}{e^{S_n}}   \leq e^{-\varepsilon\sqrt{n}\sigma} , X_n = j, Z_n>0 \right). 
\end{align*}
Fix some $A>1$. Then, for $n$ sufficiently large, such that $e^{\varepsilon \sqrt{n}}>A$,
\begin{align*}
\bb P_{i,z} &\left( \abs{ \frac{\log Z_n}{\sigma\sqrt{n}} - \frac{S_n}{\sqrt{n}\sigma}   } \geq \varepsilon , X_n = j, Z_n>0 \right) \\
&\leq \bb P_{i,z} \left( \frac{Z_n}{e^{S_n}}   \geq A , X_n = j, Z_n>0 \right) \\
&+\bb P_{i,z} \left( \frac{Z_n}{e^{S_n}}   \leq 1/A , X_n = j, Z_n>0 \right),
\end{align*}
where, by Theorem \ref{T-Loi-limite001} 
\begin{align*}
\limsup_{n\to+\infty}\sqrt{n} \bb P_{i,z} \left( \frac{Z_n}{e^{S_n}}   \geq A , X_n = j, Z_n>0 \right)
\leq  \mu_{i,z} ([A,+\infty]) \bs \nu (j) u(i,z),
\end{align*}
and 
\begin{align*}
\limsup_{n\to+\infty}\sqrt{n} \bb P_{i,z} \left( \frac{Z_n}{e^{S_n}}   \leq 1/A , X_n = j, Z_n>0 \right)
\leq \mu_{i,z} ([0,1/A]) \bs \nu (j) u(i,z).
\end{align*}
Since $\mu_{i,z}$ is a probability mesure of mass $0$ in $0$, taking the limit as $A\to +\infty$, we have that 
\begin{align*}
\lim_{n\to+\infty}\sqrt{n}\bb P_{i,z} \left( \abs{ \frac{\log Z_n}{\sigma\sqrt{n}} - \frac{S_n}{\sqrt{n}\sigma}   } \geq \varepsilon , X_n = j, Z_n>0 \right) 
=0.
\end{align*}
As $\varepsilon$ is arbitrary, using  Theorem \ref{T-Loi-limite003} we conclude the proof. 
\end{proof}


\end{document}